\documentclass[11pt]{amsart}

\usepackage{amssymb}
\usepackage{amsmath}
\usepackage{amsthm}
\usepackage{stmaryrd}
\usepackage{amsfonts,mathrsfs}
\usepackage{fullpage}
\usepackage{color}                    % For creating coloured text and background
\usepackage{hyperref}                 % For creating hyperlinks in cross references
\usepackage[all]{xy}
\title{D-bar Spark Theory and Deligne Cohomology}
\author{Ning Hao}

\begin{document}
\maketitle

\newtheorem{thm}{Theorem}[section]
\newtheorem{lem}[thm]{Lemma}
\newtheorem{prop}[thm]{Proposition}
\newtheorem{cor}[thm]{Corollary}
\newtheorem{defn}[thm]{Definition}
\newtheorem{fact}[thm]{Fact}
\newtheorem{claim}[thm]{Claim}
\newtheorem{rem}[thm]{Remark}
\newtheorem{ex}[thm]{Example}
\newtheorem{rmk}[thm]{Remark}

\begin{abstract}
We study the Harvey-Lawson spark characters of level $p$ on complex manifolds.
Presenting Deligne cohomology classes by sparks of level $p$,
we give an explicit analytic product formula for Deligne cohomology.
We also define refined Chern classes in Deligne cohomology for holomorphic vector
bundles over complex manifolds. Applications to algebraic cycles are given.
A Bott-type vanishing theorem in Deligne cohomology for holomorphic foliations is established.
A general construction of Nadel-type invariants is given together with a new proof of
Nadel's conjecture.
\end{abstract}

{\bf \tableofcontents}

\section{Introduction}
In 1970's Cheeger and Simons \cite{CS} introduced the ring of differential characters
$\hat{H}^*(X)$ associated to a smooth manifold $X$. They generalized the Chern-Weil homomorphism
and obtained a refinement of the theory of characteristic classes and characteristic forms
with applications to conformal geometry, foliation theory and more. In 2003, Harvey, Lawson
and Zweck \cite{HLZ} gave a new description of differential characters from a de Rham-Federer
viewpoint and established the ring of de Rham-Federer spark classes
which is isomorphic to the Cheeger-Simons differential characters.
Later, Harvey and Lawson expanded their approach to the theory of differential characters in \cite{HL1}.
They invented a homological apparatus to study differential characters systematically
and introduced many different presentations of differential characters. Central to their theory are
\emph{spark complexes}, \emph{sparks} and \emph{rings of spark characters} which are analogues of cochain complexes,
cocycles and cohomology rings in the usual cohomology theory.
Roughly speaking, a spark complex is a triple of cochain complexes, two of which are contained
in the main one with trivial intersection.
A spark is an element in the main complex such that its differential can be represented (uniquely)
as the sum of elements from the other two complexes. An equivalence relation among sparks is introduced
and the group of spark classes is established.
It turns out that spark complexes are abundant in geometry, topology and physics.
A more striking fact is that the classical secondary invariants like Cheeger-Simons differential characters
can be realized as the groups of spark classes associated to many different spark complexes.
Therefore, these spark complexes give many different presentations of differential characters
just as there are many different presentations of cohomology.
One basic example is the \emph{de Rham-Federer spark complex} introduced in \cite{HLZ}.
This is the complex whose objects are currents whose exterior differentials can be decomposed into
smooth forms and rectifiable currents.
Among many other examples studied in \cite{HL1}, the \emph{smooth hyperspark complex} is closely related to
$n$-gerbes with connections in physics, the \emph{Cheeger-Simons spark complex} is the closest one to the
Cheeger-Simons definition. Furthermore, all spark complexes appearing above are compatible, which implies
the groups of spark classes associated to them are all isomorphic. We call these groups the Harvey-Lawson
spark characters of a smooth manifold collectively and denote them by $\hat{\mathbf{H}}^{*}(X)$.

In \cite{HLZ}, a ring structure was constructed on the group of de Rham-Federer spark classes $\hat{\mathbf{H}}^{*}_{spark}(X)$ using a transversality theorem for currents.
A \emph{de Rham-Federer spark} is a current $\mathbf{a}$ satisfying the \emph{spark equation}
$d\mathbf{a}=\mathbf{e}-\mathbf{r}$ for some smooth differential form $\mathbf{e}$
and rectifiable current $\mathbf{r}$. If we have two spark classes $\alpha$ and $\beta$
with representatives $\mathbf{a}$ and $\mathbf{b}$ satisfying spark equations
$d\mathbf{a}=\mathbf{e}-\mathbf{r}$ and $d\mathbf{b}=\mathbf{f}-\mathbf{s}$, we may define
the product $\alpha\cdot\beta$ as the spark class represented by
$\mathbf{a}\wedge \mathbf{f}+(-1)^{\deg \mathbf{a}+1}\mathbf{r}\wedge \mathbf{b}$ which satisfies the spark equation
$d(\mathbf{a}\wedge \mathbf{f}+(-1)^{\deg \mathbf{a}+1}\mathbf{r}\wedge \mathbf{b})=
\mathbf{e}\wedge \mathbf{f}-\mathbf{r}\wedge \mathbf{s}$. But we have to worry about the
well-definedness of the wedge product of two currents. It was shown in \cite{HLZ} by
geometric measure theory that there always exist good representatives such that
all wedge products in the formula make sense and the spark class of the product is
independent of the choices of representatives. Therefore, a ring structure on $\hat{\mathbf{H}}^{*}_{spark}(X)$
is established.

In \cite{H1}, we focused on the smooth hyperspark complex and gave a construction
of the multiplication in the group of smooth hyperspark classes $\hat{\mathbf{H}}^{*}_{smooth}(X)$.
Fix a good cover $\mathcal{U}$ of $X$, a \emph{smooth hyperspark} is an element in the \v{C}ech-de Rham
double complex $$\mathbf{a}\in\bigoplus_{p+q=k}C^p(\mathcal{U},\mathcal{E}^q)$$
with the spark equation $D\mathbf{a}=\mathbf{e}-\mathbf{r}$ where
$\mathbf{e}\in\mathcal{E}^{k+1}(X)\subset C^0(\mathcal{U},\mathcal{E}^{k+1})$
and $\mathbf{r}\in C^{k+1}(\mathcal{U},\mathbb{Z})$. We introduced a cup product $\cup$
on the \v{C}ech-de Rham double complex and defined the product of two sparks
$\mathbf{a}\cdot \mathbf{b}=\mathbf{a}\cup \mathbf{f}+(-1)^{\deg \mathbf{a}+1}\mathbf{r}\cup \mathbf{b}$
for sparks $\mathbf{a}$ and $\mathbf{b}$ satisfying spark equations
$D\mathbf{a}=\mathbf{e}-\mathbf{r}$ and $D\mathbf{b}=\mathbf{f}-\mathbf{s}$. This product induces
a product in the group of smooth hyperspark classes $\hat{\mathbf{H}}^{*}_{smooth}(X)$.
Moreover, these two ring structures were shown to be compatible with the group isomorphism
$\hat{\mathbf{H}}^{*}_{spark}(X)\cong\hat{\mathbf{H}}^{*}_{smooth}(X)$.

In a recent paper \cite{HL2}, Harvey and Lawson developed a theory on $\bar{\partial}$-analogue of
differential characters for complex manifolds and introduced the Harvey-Lawson spark characters
of level $p$. While Harvey and Lawson concentrated on the spark characters of level 1 in \cite{HL2},
we generalize their theory to level $p$ for any positive integer $p$ in the first half of this paper.
To study the spark characters of level $p$,
we introduce the \emph{Dolbeault-Federer spark complex} of level $p$, the \emph{\v{C}ech-Dolbeault
spark complex} of level $p$ and the \emph{\v{C}ech-Dolbeault hyperspark complex} of level $p$
associated to a complex manifold $X$. These spark complexes are truncated versions of
the \emph{de Rham-Federer spark complex}, the \emph{smooth hyperspark complex} and the \emph{hyperspark complex}
which were first introduced in \cite{HL1} and \cite{HL2}.
It is not surprising that the groups associated to these truncated spark complexes are isomorphic to
each other. We denote these groups by $\hat{\mathbf{H}}^*(X,p)$ collectively and call them
the Harvey-Lawson spark characters of level $p$.
Furthermore, there is a group epimorphism $\Pi_p: \hat{\mathbf{H}}^*(X)\rightarrow \hat{\mathbf{H}}^*(X,p)$,
whose kernel is an ideal. Hence we establish the ring structure by identifying
$\hat{\mathbf{H}}^*(X,p)$ as the quotient ring of spark characters $\hat{\mathbf{H}}^*(X)$.

One important application of the theory of spark characters is to study analytic Deligne cohomology.
We have the following fundamental $3\times3$ grid for $\hat{\mathbf{H}}^*(X,p)$.

\xymatrix{& & & 0 \ar[d] & 0 \ar[d] & 0 \ar[d] &   \\
& & 0 \ar[r] & \frac{H^k(X,p)}{H^k_{\mathbb{Z}}(X,p)} \ar[r] \ar[d] &
\hat{\mathbf{H}}^k_{\infty}(X,p) \ar[r] \ar[d] & d_p\mathcal{E}^k(X,p) \ar[r] \ar[d] & 0 \\
& & 0 \ar[r] & H^{k+1}_{\mathcal{D}}(X,\mathbb{Z}(p)) \ar[r] \ar[d] & \hat{\mathbf{H}}^k(X,p) \ar[r]^{\delta_1}
\ar[d]^{\delta_2} & \mathcal{Z}_{\mathbb{Z}}^{k+1}(X,p) \ar[r] \ar[d] & 0 \\
& & 0 \ar[r] & \ker{\Psi_*} \ar[r]\ar[d] &
H^{k+1}(X,\mathbb{Z}) \ar[r]^{\Psi_{p*}} \ar[d] & H^{k+1}_{\mathbb{Z}}(X,p) \ar[r] \ar[d] & 0 \\
& & & 0 & 0 & 0 &\\}
It is shown in the diagram above that the analytic Deligne cohomology group
$H^k_{\mathcal{D}}(X,\mathbb{Z}(p))$ is contained in $\hat{\mathbf{H}}^{k-1}(X,p)$ as a subgroup.
Therefore, we can represent a Deligne cohomology class by a spark of level $p$.
Lifting Deligne classes to spark classes in $\hat{\mathbf{H}}^{*}(X)$ and using the
product formula for spark characters introduced earlier,
we give an explicit product formula for analytic Deligne cohomology. We also show
this product is the same as the product invented by Beilinson in \cite{B}.
Moreover, we can expand our theory to study higher operations, i.e. Massey products
for Deligne cohomology, which was introduced by Deninger \cite{De} in 1995.
We shall study Massey products in spark characters, as well as
Massey products in Deligne cohomology in \cite{H2}.

By the spark presentation of Deligne classes, it is transparent to see that
every analytic subvariety of a complex manifold represents a Deligne class.
As a direct application of our product formula for Deligne classes, we show that
the intersection of two subvarieties represents the product of their Deligne classes if
they intersect properly. In particular, when the setting is algebraic, we have a direct way
to construct the cycle map $\psi: CH^*(X)\rightarrow H^{2*}_{\mathcal{D}}(X,\mathbb{Z}(*))$ via our theory.

Cheeger and Simons \cite{CS} constructed Chern classes in differential characters for complex
vector bundles with connections which refined the usual Chern classes.
For holomorphic vector bundles over a complex manifold, we give a Chern-Weil-type construction
for Chern classes in Deligne cohomology via Cheeger-Simons theory. To construct Chern classes in Deligne cohomology
for a holomorphic vector bundle, we choose any connection compatible with the holomorphic structure,
and project the $k$th Cheeger-Simons Chern class to $\hat{\mathbf{H}}^{2k-1}(X,k)$ whose image is actually
in $H^{2k}_{\mathcal{D}}(X,\mathbb{Z}(k))$. We show that the image of $k$th Cheeger-Simons Chern in $H^{2k}_{\mathcal{D}}(X,\mathbb{Z}(k))$ is independent of the choice of connection and define it as our
$k$th Chern class. The functorial property of Chern classes and the Whitney formula are shown as well.
In \cite{Z}, Zucker indicated that the splitting principle works well in defining Chern classes
in Deligne cohomology. In contrast to Zucker's method, our method is constructive since it is possible to
explicitly construct representatives of Cheeger-Simons Chern classes via methods of Harvey-Lawson \cite{HL3}
or Brylinski-McLaughlin \cite{BrM}.

In 1969, Bott \cite{Bo} constructed a family of connections on the normal bundle of any smooth
foliation of a manifold and established the Bott vanishing theorem which says
the characteristic classes of the normal bundle are trivial in all sufficiently high degrees.
We prove an analogue of the Bott vanishing theorem for our refined Chern classes of the normal
bundle of a holomorphic foliation.

In 1997, Nadel \cite{N} introduced relative invariants for holomorphic
vector bundles. Explicitly, for two holomorphic vector bundles $E$ and $F$ over a
complex manifold $X$ which are $C^{\infty}$ isomorphic, Nadel
defined invariants $\mathscr{E}^k(E,F)\in H^{2k-1}(X,\mathcal{O})/H^{2k-1}(X,\mathbb{Z})$.
He also conjectured that these invariants should coincide with a component
of the Abel-Jacobi image of $k!(ch_k(E)-ch_k(F))\in CH^k_{hom}(X)$ when the setting is algebraic.
This conjecture was proved by Berthomieu \cite{Be} by his relative K-theory.
Since we define Chern classes and Chern characters in Deligne cohomology for holomorphic vector bundles,
we can construct Nedel-type invariants $\hat{\mathscr{E}}^k(E,F)$ in intermediate Jacobians
in more general setting ( not necessarily algebraic ).
Moreover, we show that $\hat{\mathscr{E}}^k(E,F)$ is represented by a smooth $2k-1$-form whose
$(0, 2k-1)$ component represents $\mathscr{E}^k(E,F)\in H^{2k-1}(X,\mathcal{O})/H^{2k-1}(X,\mathbb{Z})$.
Therefore, Nadel's conjecture is proved in a more general context.

The organization of this paper is the following. First we recall the definition and basic properties
of (generalized) spark complex invented by Harvey and Lawson \cite{HL1}
\cite{HL2}. The main examples --- the Dolbeault-Federer spark complex, the \v{C}ech-Dolbeault
spark complex, the \v{C}ech-Dolbeault hyperspark complex --- are introduced in Section 3-5.
The ring structure and functoriality of spark characters associated to these spark complexes
are established and the equivalence of them are also verified. Using product formula of
spark characters introduced in \cite{HLZ} \cite{H1}, we give an explicit product formula
for analytic Deligne cohomology in Section 6. We apply spark theory to algebraic cycles and
define the ring homomorphism from Chow ring to Deligne cohomology in Section 7.
In Section 8, we define the refined Chern classes in Deligne cohomology for holomorphic vector
bundles over complex manifolds. And in Section 9, we prove an analogue of the Bott vanishing theorem
for holomorphic foliations in this context. In Section 10, we define Nadel-type invariants for
holomorphic vector bundles in the intermediate Jacobians and prove Nadel's conjecture.

\noindent{\textbf{Acknowledgements.}} I am very grateful to my advisor H. Blaine Lawson
for introducing this subject to me. I am also indebted to him for his encouragement and helpful discussions during the preparation of this paper. I would like to thank Zhiwei Yun and Li Li for useful discussions. 

\section{Generalized Spark Complexes}

We follow \cite{HL2} to give the definitions of a generalized spark complex and its associated
group of spark classes which are generalizations of the spark complex and its associated group of
spark classes appeared in \cite{HL1} \cite{H1}. When we mention spark complex in this
paper, we mean this generalized spark complex defined below.

\begin{defn}
A (generalized) \textbf{homological spark complex}, or \textbf{spark complex} for short,
is a triple of cochain complexes $(F^*,E^*,I^*)$ together with morphisms
$$I^* \stackrel{\Psi}\rightarrow F^* \hookleftarrow E^*$$
such that
\begin{enumerate}
            \item $I^k\cap E^k={0}$ for $k>0$, $F^k=E^k=I^k=0$ for $k<0$,
            \item $H^*(E^*)\cong H^*(F^*)$,
            \item $\Psi|_{I^0}: I^0 \rightarrow F^0$ is injective.
          \end{enumerate}
\end{defn}
\begin{defn}
In a given spark complex $(F^*,E^*,I^*)$, a \textbf{spark} of degree $k$ is
a pair $(a,r)\in F^k \oplus I^{k+1}$ which satisfies the \textbf{spark equations}
\begin{enumerate}
  \item $da=e-\Psi(r)$ for some $e\in E^{k+1}$,
  \item $dr=0$.
\end{enumerate}

Two sparks $(a,r),(a',r')$ of degree $k$ are \textbf{equivalent}
if there exists a pair $(b,s)\in F^{k-1}\oplus I^k$ such that
\begin{enumerate}
  \item $a-a'=db+\Psi(s)$,
  \item $r-r'=-ds$.
\end{enumerate}

The set of equivalence classes is called the group of spark classes
of degree $k$ and denoted by $\hat{\mathbf{H}}^k(F^*,E^*,I^*)$ or $\hat{\mathbf{H}}^k$
for short. Let $[(a,r)]$ denote the equivalence class containing the spark $(a,r)$.
\end{defn}

\begin{rmk}
Harvey and Lawson introduced spark complexes in \cite{HL1} where they
required $\Psi$ to be injective. In that case, $\mathbf{e}$ and $\mathbf{r}$
are uniquely determined by $\mathbf{a}$.
Later, they generalized the original definition and defined the generalized
spark complex in \cite{HL2}
where $\Psi: I^*\rightarrow F^*$ was not required to be injective.
Hence, $\mathbf{r}$ is not determined uniquely by $\mathbf{a}$
and we have to remember $\mathbf{r}$ for a spark
and denote a spark by $(\mathbf{a},\mathbf{r})$.

Also, in this paper, when we discuss spark complexes
in which $\Psi$ is injective, we may denote a spark only by $\mathbf{a}$ and omit $\mathbf{r}$.
\end{rmk}

We now derive the fundamental exact sequences associated to a spark complex
$(F^*,E^*,I^*)$. Let $Z^k(E^*)=\{e\in E^k: de=0\}$ and set
$$Z_I^k(E^*)\equiv \{e\in Z^k(E^*): [e]=\Psi_*([r]) \text{ for some } [r]\in H^k(I^*)\}$$
where $[e]$ denotes the class of $e$ in $H^k(E^*)\cong H^k(F^*)$.

\begin{lem}
There exist well-defined surjective homomorphisms
$$\delta_1: \hat{\mathbf{H}}^k\rightarrow Z_I^{k+1}(E^*) \quad \text{and} \quad
\delta_2: \hat{\mathbf{H}}^k \rightarrow H^{k+1}(I^*)$$ given by
$$\delta_1([(a,r)])=e \quad \text {and} \quad \delta_2([(a,r)])=[r]$$ where $da=e-\Psi(r)$.
\end{lem}

\begin{proof}
If $(a',r')$ is equivalent to $(a,r)$, i.e. $a-a'=db+\Psi(s)$ and $r-r'=-ds$, then we have
$da'=e-\Psi(r+ds)$. So it is easy to see these maps are well-defined.

Consider $e\in Z_I^{k+1}(E^*)$, by definition, there exists $r\in I^{k+1}$ such that
$e-\Psi(r)$ is exact in $F^{k+1}$, i.e. $\exists a\in F^k$ with $da=e-\Psi(r)$. So $\delta_1([(a,r)])=e$.
For $[r]\in H^{k+1}(I^*)$, $\Psi(r)$ also represents a class in $H^{k+1}(F^*)\cong H^{k+1}(E^*)$.
Choosing a representative $e \in E^{k+1}$ of this class, we have $e-\Psi(r)=da$ for some
$a\in F^k$, hence $\delta_2([(a,r)])=[r]$. Both $\delta_1$ and $\delta_2$ are surjective.

\end{proof}

\begin{lem}
Define $\hat{\mathbf{H}}^k_E\equiv \ker\delta_2$, then $\hat{\mathbf{H}}^k_E\cong E^k/Z_I^{k}(E^*)$.
\end{lem}

\begin{proof}
Let $\alpha\in\hat{\mathbf{H}}^k_E$ be represented by $(a,r)$ with spark equations $da=e-\Psi(r)$
and $dr=0$. Then we have  $[r]=\delta_2(\alpha)=0$, i.e. $r=-ds$ for some $s\in I^k$. So
$d(a-\Psi(s))=e$, by \cite[Lemma 1.5]{HL1} and the fact $H^*(F^*)\cong H^*(E^*)$, there exists
$b\in F^{k-1}$ such that $a'\equiv a-\Psi(s)+db\in E^k$. Hence $\alpha$ can be represented by
spark $(a',0)$ with $a'\in E^k$. If $(a',0)$ is equivalent to 0, then $a'=db'+\Psi(s')$ for some
$b'\in F^{k-1}$ and $s'\in I^k$ with $ds'=0$, i.e. $a'\in Z_I^{k}(E^*)$.
\end{proof}

\begin{rmk}
From last proof, it is easy to see that  $\hat{\mathbf{H}}^k_E$ is the space of spark classes
that can be represented by sparks of type $(a,0)$ where $a\in E^k$.
\end{rmk}

\begin{defn}
Associated to any spark complex $(F^*,E^*,I^*)$ is the cone complex $(G^*,D)$ defined by setting
$$G^k\equiv F^k\oplus I^{k+1} \text{ with differential } D(a,r)=(da+\Psi(r),-dr).$$
\end{defn}

Consider the homomorphism $\Psi_*:H^k(I^*)\rightarrow H^k(F^*)=H^k(E^*)$, and define
$$H^k_I(E^*)\equiv \text{Image}\{\Psi_*\} \quad \text{and} \quad Ker^k(I^*)\equiv \ker\{\Psi_*\}.$$

\begin{prop}\cite{HL2}
There are two fundamental short exact sequences
\begin{enumerate}
  \item $0\longrightarrow H^k(G^*)\longrightarrow \hat{\mathbf{H}}^k \stackrel{\delta_1}\longrightarrow
Z_I^{k+1}(E^*)\longrightarrow 0$;
  \item $0\longrightarrow \hat{\mathbf{H}}^k_E \longrightarrow \hat{\mathbf{H}}^k \stackrel{\delta_2}\longrightarrow
H^{k+1}(I^*)\longrightarrow 0$.
\end{enumerate}

Moreover, associated to any spark complex $(F^*,E^*,I^*)$ is the commutative diagram

\xymatrix{& & & 0 \ar[d] & 0 \ar[d] & 0 \ar[d] &   \\
& & 0 \ar[r] & \frac{H^k(E^*)}{H^k_I(E^*)} \ar[r] \ar[d] &
\hat{\mathbf{H}}^k_E \ar[r] \ar[d] & dE^k \ar[r] \ar[d] & 0 \\
& & 0 \ar[r] & H^k(G^*) \ar[r] \ar[d] & \hat{\mathbf{H}}^k \ar[r]^{\delta_1}
\ar[d]^{\delta_2} & Z_I^{k+1}(E^*) \ar[r] \ar[d] & 0 \\
& & 0 \ar[r] & Ker^{k+1}(I^*) \ar[r]\ar[d] &
H^{k+1}(I^*) \ar[r] \ar[d] & H_I^{k+1}(E^*) \ar[r] \ar[d] & 0 \\
& & & 0 & 0 & 0 &\\}
whose rows and columns are exact.
\end{prop}

Also, we can talk about quasi-isomorphism between two (generalized) spark complexes.

\begin{defn}
Two spark complexes $(F^*,E^*,I^*)$ and $(\bar{F}^*,\bar{E}^*,\bar{I}^*)$
are quasi-isomorphic if there exists a commutative diagram of morphisms

\xymatrix{& & & & & I^* \ar[r]^{\Psi} \ar[d]^{\psi} &
F^* \ar@{^{(}->}[d]^i & E^*  \ar @{_{(}->}[l]_i \ar @{=}[d] \\
& & & & & \bar{I}^* \ar[r]^{\bar{\Psi}}  & \bar{F}^*
 & \bar{E}^* \ar @{_{(}->}[l]_i\\}
inducing an isomorphism $$\psi^*: H^*(I^*) \stackrel{\cong}\longrightarrow H^*(\bar{I}^*).$$

\end{defn}

\begin{prop}\cite{HL2}
A quasi-isomorphism of spark complexes $(F^*,E^*,I^*)$ and $(\bar{F}^*,\bar{E}^*,\bar{I}^*)$
induces an isomorphism
$$\hat{\mathbf{H}}^*(F^*,E^*,I^*)\cong \hat{\mathbf{H}}^*(\bar{F}^*,\bar{E}^*,\bar{I}^*)$$
of the associated groups of spark classes. Moreover, it induces an isomorphism of the
$3\times 3$ grids associated to these two complexes.
\end{prop}

\section{Dolbeault-Federer Sparks of Level p}

Let $X$ be a complex manifold. Recall the \textbf{de Rham-Federer spark complex} \cite{HLZ} associated to $X$
is a spark complex $(F^*,E^*,I^*)$ where
$$F^*\equiv \mathcal{D}'^*(X), E^*\equiv \mathcal{E}^*(X), I^*\equiv \mathcal{IF}^*(X).$$
Note that $\mathcal{E}^*$ and $\mathcal{D}'^*$ denote
the sheaves complex-valued smooth forms and currents respectively.
And $\mathcal{IF}^*$ is the sheaf of locally integrally flat currents on $X$.
The associated group of spark classes is denoted by $\hat{\mathbf{H}}^*_{spark}(X)$, or
$\hat{\mathbf{H}}^*(X)$ for short.
In fact, $\hat{\mathbf{H}}^*_{spark}(X)$ is a ring and functorial with respect to
smooth map between manifolds. We refer to \cite{HL1} \cite{HLZ} for details.

Now we introduce a new spark complex, the Dolbeault-Federer spark complex of level $p$,
which is closely related to the de Rham-Federer spark complex.

For a complex manifold $X$, we can decompose the space of smooth $k$-forms by types:
$$\mathcal{E}^k(X) \equiv \bigoplus_{r+s=k}\mathcal{E}^{r,s}(X).$$
And similarly,
$$\mathcal{D}'^k(X) \equiv \bigoplus_{r+s=k}\mathcal{D}'^{r,s}(X).$$

Fix an integer $p>0$ and consider the truncated complex $(\mathcal{D}'^*(X,p),d_p)$ with
$$\mathcal{D}'^k(X,p) \equiv \bigoplus_{r+s=k,r<p}\mathcal{D}'^{r,s}(X) \textrm{ and } d_p\equiv
\pi_p\circ d $$ where $\pi_p: \mathcal{D}'^k(X)\rightarrow\mathcal{D}'^k(X,p)$ is the natural
projection $\pi_p(a)=a^{0,k}+...+a^{p-1,k-p+1}$. Similarly, we can define $(\mathcal{E}^*(X,p),d_p)$.

\begin{defn}By the \textbf{Dolbeault-Federer spark complex of level} $\mathbf{p}$, or more simply,
the $\bar{\mathbf{d}}$\textbf{-spark complex of level} $\mathbf{p}$ we mean the triple $(F^*_p,E^*_p,I^*_p)$
$$F^*_p \equiv \mathcal{D}'^*(X,p),\quad E^*_p \equiv \mathcal{E}^*(X,p),\quad I^*_p \equiv \mathcal{IF}^*(X)$$
 with maps $$E^*_p \hookrightarrow F^*_p \quad\textrm{and}\quad \Psi_p:I^*_p \rightarrow F^*_p$$
 where $\Psi_p=\pi_p\circ i$.
\end{defn}

\begin{rmk}
The triple $(F^*_p,E^*_p,I^*_p)\equiv (\mathcal{D}'^*(X,p),\mathcal{E}^*(X,p),\mathcal{IF}^*(X))$
is a spark complex.
\end{rmk}
\begin{proof}
First, $$H^*(F^*_p)\cong H^*(E^*_p)\cong \mathbb{H}^*(\Omega^{*<p})\equiv H^*(X,p).$$
where $\mathbb{H}^*(\Omega^{*<p})\equiv H^*(X,p)$ denotes the hypercohomology of complex of sheaves
$$0\rightarrow\Omega^0\rightarrow\Omega^1\rightarrow\Omega^2\rightarrow\cdots
\rightarrow\Omega^{p-1}\rightarrow0$$
and $\Omega^k$ is the sheaf of holomorphic $k$-forms on $X$.

For the proof of $\Psi_p(I^k_p)\cap E^k_p=\{0\}$ for
$k>0$, we refer to \cite[Appendix B]{HL2}.
\end{proof}

\begin{defn}
A \textbf{Dolbeault-Federer spark of level} $\mathbf{p}$ of degree $k$,
or a $\bar{\mathbf{d}}$-\textbf{spark of level} $\mathbf{p}$ is a pair
$$(a,r)\in\mathcal{D}'^k(X,p)\oplus\mathcal{IF}^{k+1}(X)$$
satisfying the spark equations $$d_pa=e-\Psi_p(r)\quad \text{and}\quad dr=0$$
for some $e\in\mathcal{E}^{k+1}(X,p)$.

Two Dolbeault-Federer sparks of level $p$, $(a,r)$ and $(a',r')$
are \textbf{equivalent} if there exist
$b\in\mathcal{D}'^{k-1}(X,p)$ and $s\in \mathcal{IF}^k(X)$
such that $$a-a'=d_pb+\Psi_p(s) \quad\text{and}\quad r-r'=-ds.$$

The equivalence class determined by a spark $(a,r)$
will be denoted by $[(a,r)]$, and the group of
Dolbeault-Federer spark classes of level $p$ of degree $k$
will be denoted by $\hat{\mathbf{H}}^k_{spark}(X,p)$ or $\hat{\mathbf{H}}^k(X,p)$ for short.

\end{defn}

Applying Proposition 2.8, we have

\begin{prop}
Let $H_{\mathbb{Z}}^{k+1}(X,p)$ denote the image of map $\Psi_{p*}: H^{k+1}(X,\mathbb{Z})\rightarrow
H^{k+1}(X,p)$, and
$\mathcal{Z}_{\mathbb{Z}}^{k+1}(X,p)$ denote the set of $d_p$-closed forms in
$\mathcal{E}^{k+1}(X,p)$ which represent classes in $H_{\mathbb{Z}}^{k+1}(X,p)$.
Let $\hat{\mathbf{H}}^k_{\infty}(X,p)$ denote the spark classes representable by smooth forms,
and $H^{k+1}_{\mathcal{D}}(X,\mathbb{Z}(p))$ denote the Deligne cohomology group.

The $3\times 3$ diagram for $\hat{\mathbf{H}}^k(X,p)$ can be written as

\xymatrix{& & & 0 \ar[d] & 0 \ar[d] & 0 \ar[d] &   \\
& & 0 \ar[r] & \frac{H^k(X,p)}{H^k_{\mathbb{Z}}(X,p)} \ar[r] \ar[d] &
\hat{\mathbf{H}}^k_{\infty}(X,p) \ar[r] \ar[d] & d_p\mathcal{E}^k(X,p) \ar[r] \ar[d] & 0 \\
& & 0 \ar[r] & H^{k+1}_{\mathcal{D}}(X,\mathbb{Z}(p)) \ar[r] \ar[d] & \hat{\mathbf{H}}^k(X,p) \ar[r]^{\delta_1}
\ar[d]^{\delta_2} & \mathcal{Z}_{\mathbb{Z}}^{k+1}(X,p) \ar[r] \ar[d] & 0 \\
& & 0 \ar[r] & \ker{\Psi_*} \ar[r]\ar[d] &
H^{k+1}(X,\mathbb{Z}) \ar[r]^{\Psi_{p*}} \ar[d] & H^{k+1}_{\mathbb{Z}}(X,p) \ar[r] \ar[d] & 0 \\
& & & 0 & 0 & 0 &\\}

A special and the most interesting case is when $X$ is K\"ahler and $k=2p-1$,

\xymatrix{& & & 0 \ar[d] & 0 \ar[d] & 0 \ar[d] &   \\
& & 0 \ar[r] & \mathcal{J}^p(X) \ar[r] \ar[d] &
\hat{\mathbf{H}}^{2p-1}_{\infty}(X,p) \ar[r] \ar[d] & d_p\mathcal{E}^{2p-1}(X,p) \ar[r] \ar[d] & 0 \\
& & 0 \ar[r] & H^{2p}_{\mathcal{D}}(X,\mathbb{Z}(p)) \ar[r] \ar[d] & \hat{\mathbf{H}}^{2p-1}(X,p) \ar[r]^{\delta_1}
\ar[d]^{\delta_2} & \mathcal{Z}_{\mathbb{Z}}^{2p}(X,p) \ar[r] \ar[d] & 0 \\
& & 0 \ar[r] & \text{Hdg}^{p,p}(X) \ar[r]\ar[d] &
H^{2p}(X,\mathbb{Z}) \ar[r]^{\Psi_{p*}} \ar[d] & H^{2p}_{\mathbb{Z}}(X,p) \ar[r] \ar[d] & 0 \\
& & & 0 & 0 & 0 &\\}
where $\mathcal{J}^p(X)$ denotes the $p$th intermediate Jacobian and
$\text{Hdg}^{p,p}(X)$ is the set of the Hodge classes.

\end{prop}

\begin{proof}
The proof follows Proposition 2.8 directly. The only nontrivial part is why Deligne cohomology appears
in the middle row. We postpone our proof to Section 6 where we study Deligne cohomology in detail.
\end{proof}

\begin{rmk}
The $\bar{d}$-spark complex is a generalization
of $\bar{\partial}$-spark complex in \cite{HL2} which corresponds the special case $p=1$.
\end{rmk}

\subsection{Ring Structure}
We can establish the ring structure on $ \hat{\mathbf{H}}^*(X,p)$ by
identifying it as a quotient ring of $ \hat{\mathbf{H}}^*(X)$.

Consider the following commutative diagram:

 \xymatrix{& & & \mathcal{IF}^*(X)
\ar[r]^i \ar[d]^{id} & \mathcal{D}'^*(X)
 \ar[d]^{\pi_p} & \mathcal{E}^*(X) \ar[l]_i \ar[d]^{\pi_p} \\
& & & \mathcal{IF}^*(X) \ar[r]^{\pi_p \circ i} & \mathcal{D}'^*(X,p) &
\mathcal{E}^*(X,p) \ar[l]_i }
which induces a group homomorphism $\Pi_p: \hat{\mathbf{H}}^*(X)\rightarrow\hat{\mathbf{H}}^*(X,p)$.
Furthermore, we have

\begin{thm}
The morphism of spark complexes $(\pi_p, \pi_p, id): (F^*, E^*, I^*)\rightarrow(F^*_p, E^*_p, I^*_p)$
induces a surjective group homomorphism $$\Pi_p:\hat{\mathbf{H}}^*(X) \rightarrow
\hat{\mathbf{H}}^*(X,p)$$ whose kernel is an ideal. Hence, $\hat{\mathbf{H}}^*(X,p)$ carries a ring structure.
\end{thm}

\begin{proof}
It's straightforward to see that the diagram above commutes and $\pi_p$ commutes with differentials.
Consequently, the induced map $(a,r)\mapsto(\pi_p(a),r)$ on sparks descends to a well defined
homomorphism $\Pi_p:\hat{\mathbf{H}}^k(X)\rightarrow\hat{\mathbf{H}}^k(X,p)$ as claimed.

To prove the surjectivity, consider a spark $(A,r)\in F^k_p\oplus I^{k+1}$ with $dr=0$ and $d_pA=e-\Psi_p(r)$
for some $e\in E^{k+1}_p$. We can choose some smooth form which represents same cohomology class with $r$ in
$H^{k+1}(F^*)\cong H^{k+1}(E^*)$, so there exist $a_0\in F^k$, $e_0\in E^{k+1}$ such that $da_0=e_0-r$.
We have $\pi_p (da_0)=\pi_p(e_0-r)\Rightarrow d_p(\pi_p a_0)=\pi_p e_0-\Psi_p r$. Hence,
$d_p(A-\pi_p a_0)=e-\pi_pe_0$ is a smooth form.  It follows by \cite[Lemma 1.5]{HL1} that there exist
$b\in F^{k-1}_p$ and $f\in E^k_p$ with $A-\pi_pa_0=f+d_pb$. Set $a=a_0+f+db$ and note that
$da=da_0+df+ddb=e_0-r+df=(e_0+df)-r$. Hence, $(a,r)$ is a spark of degree $k$ and
$\pi_pa=\pi_p(a_0+f+db)=\pi_pa_0+f+d_pb=A$. So $\Pi_p$ is surjective.

We need the following lemma to show the kernel is an ideal.

\begin{lem}
On $\hat{\mathbf{H}}^k(X)$, one has that $\ker(\Pi_p)=\{\alpha \in
\hat{\mathbf{H}}^k(X) : \exists (a,0) \in \alpha$ where a is
smooth and $\pi_p(a)=0\}$. In particular, $\ker(\Pi_p)\subset
\hat{\mathbf{H}}^k_{\infty}(X)$.
\end{lem}
\begin{proof}
One direction is clear. Suppose $\alpha \in \ker(\Pi_p)$ and choose
any spark $(a,r) \in \alpha$. $\Pi_p(\alpha)=0$ means that there exist $
b \in F^{k-1}_p$ and $s\in I^k$ with
$$\left\{%
\begin{array}{ll}
    \pi_p(a)=d_pb+\Psi_p (s)=\pi_p(db+s) \\
    r=-ds \\
\end{array}%
\right.$$
Replace $(a,r)$ by $(\tilde{a},0)=(a-db-s,r+ds)$, note
that $\pi_p(\tilde{a})=\pi_p(a-db-s)=0$.

In fact, we can choose $\tilde{a}$ to be smooth. $d\tilde{a}=da-ds=e-r-ds=e$
is a smooth form, it follows by \cite[Lemma 1.5]{HL1} and the fact
$H^*(F^p\mathcal{D}'^*(X))=H^*(F^p\mathcal{E}^*(X))$ that we can choose
$\tilde{a}$ to be smooth. Note that $F^0\mathcal{D}'^*\supset F^1\mathcal{D}'^*\supset
\cdots \supset F^p\mathcal{D}'^*\supset \cdots$ is the naive filtration.
\end{proof}
By the product formula of $\hat{\mathbf{H}}^*(X)$ in \cite{HLZ},
it is easy to see the kernel is an ideal.
In fact, if $\alpha$ and $\beta$ are two spark classes, and $\alpha\in\ker(\Pi_p)$,
then we can choose representatives $(a,0)$ and $(b,s)$ for $\alpha$ and $\beta$
respectively, with spark equations $da=e-0$ and $db=f-s$, where $a$, $e$, $f$ are smooth.
By the product formula, $\alpha\beta$ can be represented by
$(a\wedge f+(-1)^{\deg a+1}0\wedge b,0\wedge s)=(a\wedge f,0)$ which is in $\ker(\Pi_p)$.

Hence, $\hat{\mathbf{H}}^*(X,p)$ carries a ring structure induced from
$\hat{\mathbf{H}}^*(X)$.

\end{proof}

\subsection{Functoriality}
\begin{prop}
There are commutative diagrams

 \xymatrix{
 & & \hat{\mathbf{H}}^k(X) \ar[r]^{\delta_1} \ar[d]^{\Pi_p}
 & \mathcal{Z}^{k+1}_{\mathbb{Z}}(X) \ar[d]^{\pi_p} &
 & \hat{\mathbf{H}}^k(X) \ar[r]^{\delta_2} \ar[d]^{\Pi_p} & H^{k+1}(X,\mathbb{Z}) \ar[d]^=\\
 & & \hat{\mathbf{H}}^k(X,p) \ar[r]^{\delta_1} &  \mathcal{Z}^{k+1}_{\mathbb{Z}}(X,p) &
 & \hat{\mathbf{H}}^k(X,p) \ar[r]^{\delta_2} & H^{k+1}(X,\mathbb{Z})}
\end{prop}

\begin{proof}
Let $\alpha\in\hat{\mathbf{H}}^k(X)$. Choose a representative $(a,r)\in \alpha$ with
spark equation $da=e-r$. Then $\pi_p\circ\delta_1(\alpha)=\pi_p(e)$, and
$\delta_1\circ\Pi_p(\alpha)=\delta_1\circ\Pi_p([(a,r)])=\delta_1([(\pi_p(a),r)])=\pi_p(e)$
since $(\pi_p(a),r)$ is a $\bar{d}$-spark of level $p$ with spark equation
$d_p(\pi_pa)=\pi_p(e)-\Psi_pr$. Hence, the first diagram is commutative.
We can verify the second one by the same way.

\end{proof}

Moreover, we have the following theorem
\begin{thm}
Any holomorphic map $f:X \rightarrow Y$ between complex manifolds induces
a graded ring homomorphism
$$f^*:\hat{\mathbf{H}}^*(Y,p) \rightarrow \hat{\mathbf{H}}^*(X,p)$$
with the property that if $g: Y \rightarrow Z$ is holomorphic,
then $(g\circ f)^*=f^*\circ g^*$.

\end{thm}

\begin{proof}
It was shown in \cite{HLZ} \cite{HL1} that $f$ induces a ring homomorphism
$f^*:\hat{\mathbf{H}}^*(Y) \rightarrow \hat{\mathbf{H}}^*(X)$ with the asserted property.
It suffices to show $f^*(\ker\Pi_p)\subset(\ker\Pi_p)$ which is directly from Lemma 3.7.
\end{proof}

\begin{cor}
$\hat{\mathbf{H}}^*(\bullet,p)$ is a graded ring functor on the category of complex
manifolds and holomorphic maps.
\end{cor}
\begin{thm}
(Gysin map) Any holomorphic map $f:X^{m+r} \rightarrow Y^m$ between complex manifolds induces
a graded ring homomorphism
$$f_*:\hat{\mathbf{H}}^*(X,p) \rightarrow \hat{\mathbf{H}}^{*-2r}(Y,p-r).$$
\end{thm}
\begin{proof}
It was shown in \cite{HLZ} \cite{HL1} that $f$ induces a Gysin map
$f_*:\hat{\mathbf{H}}^*(X) \rightarrow \hat{\mathbf{H}}^{*-r}(Y)$.
Moreover, from Lemma 3.7, it is plain to get $f^*(\ker\Pi_p)\subset(\ker\Pi_{p-r})$.
\end{proof}

\section{\v{C}ech-Dolbeault Sparks of Level p}
We now consider other presentations of the $\bar{d}$-spark classes.
We shall introduce the \v{C}ech-Dolbeault spark complex of level $p$ which is a generalization of
the \v{C}ech-Dolbeault spark complex in \cite{HL2}.

Recall that, for a complex manifold $X$, we can decompose the space of smooth $k$-forms
over an open set $U\subset X$ by types:
$$\mathcal{E}^k(U) \equiv \bigoplus_{r+s=k}\mathcal{E}^{r,s}(U).$$
Let $\mathcal{E}^k(U,p)=\bigoplus_{r+s=k,r<p}\mathcal{E}^{r,s}(U)$, and
$\mathcal{E}^k_p$ denote the subsheaf of $\mathcal{E}^k$ with
$\mathcal{E}^k_p(U)=\mathcal{E}^k(U,p)$.
And similarly, we can define the sheaf $\mathcal{D}'^k_p$ with
$\mathcal{D}'^k_p(U)=\mathcal{D}'^k(U,p)$.

Suppose $\mathcal{U}$ is a good cover of $X$ and consider the total complex of
the following double complex with
total differential $D_p=\delta+(-1)^rd_p$:

\xymatrix{
\vdots  & \vdots  &\vdots  &
 & \vdots  &\\
C^0(\mathcal{U},\mathcal{E}_p^2) \ar[r]^{\delta}\ar[u]^{d_p}  & C^1(\mathcal{U},\mathcal{E}_p^2) \ar[r]^{\delta} \ar[u]^{-d_p}
& C^2(\mathcal{U},\mathcal{E}_p^2) \ar[r]^{\delta} \ar[u]^{d_p} & \cdots\quad\cdots \ar[r]^{\delta} &
C^r(\mathcal{U},\mathcal{E}_p^2) \ar[u]^{(-1)^rd_p} \ar[r]^{\delta} & \cdots  \\
C^0(\mathcal{U},\mathcal{E}_p^1) \ar[r]^{\delta}\ar[u]^{d_p}  & C^1(\mathcal{U},\mathcal{E}_p^1) \ar[r]^{\delta} \ar[u]^{-d_p}
& C^2(\mathcal{U},\mathcal{E}_p^1) \ar[r]^{\delta} \ar[u]^{d_p} & \cdots\quad\cdots \ar[r]^{\delta} &
C^r(\mathcal{U},\mathcal{E}_p^1) \ar[u]^{(-1)^rd_p} \ar[r]^{\delta} & \cdots   \\
C^0(\mathcal{U},\mathcal{E}_p^0) \ar[r]^{\delta}\ar[u]^{d_p}  & C^1(\mathcal{U},\mathcal{E}_p^0) \ar[r]^{\delta} \ar[u]^{-d_p}
& C^2(\mathcal{U},\mathcal{E}_p^0) \ar[r]^{\delta} \ar[u]^{d_p} & \cdots\quad\cdots \ar[r]^{\delta} &
C^r(\mathcal{U},\mathcal{E}_p^0) \ar[u]^{(-1)^rd_p} \ar[r]^{\delta} & \cdots    \\
}

It is easy to see the row complexes are exact everywhere except in the first column on the left, and
$$\{\ker(\delta) \text{ on the left column}\}\cong\{\text{global sections of sheaves } \mathcal{E}_p^*\}
= \mathcal{E}^*(X,p).$$
Hence, $$H^*(\bigoplus_{r+s=*}C^r(\mathcal{U},\mathcal{E}_p^s))\cong H^*(\mathcal{E}^*(X,p))\cong H^*(X,p).$$

Note that every column complex is exact everywhere except at the bottom and the level of $p$ from the bottom.

Now we consider the triple of complexes
$$(F^*_p,E^*_p,I^*_p)\equiv(\bigoplus_{r+s=*}C^r(\mathcal{U},\mathcal{E}_p^s),
\mathcal{E}^*(X,p),C^*(\mathcal{U},\mathbb{Z})).$$
And we have
\begin{prop}
The triple $(F^*_p,E^*_p,I^*_p)$ defined above is a spark complex ( even in the sense of \cite{HL1} ),
which is called the \textbf{\v{C}ech-Dolbeault spark complex of level} $\mathbf{p}$,
or the \textbf{smooth hyperspark complex of level} $\mathbf{p}$.
\end{prop}
\begin{proof}
We have shown that $E^*_p\hookrightarrow F^*_p$ induces an isomorphism
$H^*(E^*_p)\cong H^*(F^*_p)$.
Also there is an injective cochain map $I^*_p\equiv C^*(\mathcal{U},\mathbb{Z}) \hookrightarrow
C^*(\mathcal{U},\mathcal{E}_p^0) \hookrightarrow \bigoplus_{r+s=*}C^r(\mathcal{U},\mathcal{E}_p^s)
\equiv F^*_p$.

$E^k_p\cap I^k_p=\{0\}$ for $k>0$ is trivial.
\end{proof}

\begin{defn}
A \textbf{\v{C}ech-Dolbeault spark of level} $\mathbf{p}$ of degree $k$,
or a \textbf{smooth hyperspark of level} $\mathbf{p}$
is an element $$a\in\bigoplus_{r+s=k}C^r(\mathcal{U},\mathcal{E}_p^s)$$
with the spark equation $$D_pa=e-r$$ where $e\in\mathcal{E}_p^{k+1}(X)\subset C^0(\mathcal{U},\mathcal{E}_p^{k+1})$
is of bidegree $(0,k+1)$ and $r\in C^{k+1}(\mathcal{U},\mathbb{Z})$.

Two \v{C}ech-Dolbeault sparks of level $p$, $a$ and $a'$ are \textbf{equivalent} if there exist $b\in
\bigoplus_{r+s=k-1}C^r(\mathcal{U},\mathcal{E}_p^s)$ and $s\in C^k(\mathcal{U},\mathbb{Z})$
satisfying $$a-a'=D_pb+s.$$

The equivalence class determined by a \v{C}ech-Dolbeault spark $a$ will be denoted by $[a]$, and the group of
\v{C}ech-Dolbeault spark classes of level $p$ will be denoted by $\hat{\mathbf{H}}^k_{smooth}(X,p)$.

\end{defn}

Recall that the \textbf{smooth hyperspark complex} (\cite{HL1} \cite{H1}) is defined by
$$(F^*,E^*,I^*)=(\bigoplus_{r+s=*}C^r(\mathcal{U},\mathcal{E}^s),\mathcal{E}^*(X),
C^*(\mathcal{U},\mathbb{Z})).$$
The associated group of smooth hyperspark classes is denoted by $\hat{\mathbf{H}}^*_{smooth}(X)$,
whose ring structure was established in \cite{H1}.
The relation between
the smooth hyperspark complex and the \v{C}ech-Dolbeault spark complex of level $p$ is the same as
the relation between the de Rham-Federer spark complex and the Dolbeault-Federer spark complex of level
$p$. We have the natural morphism $(\pi_p,\pi_p,id):
(F^*,E^*,I^*)\longrightarrow (F^*_p,E^*_p,I^*_p)$. Explicitly, we have the following
commutative diagram

 \xymatrix{& & & C^*(\mathcal{U},\mathbb{Z})
\ar[r]^{i\qquad} \ar[d]^{id} & \bigoplus_{r+s=*}C^r(\mathcal{U},\mathcal{E}^s)
 \ar[d]^{\pi_p} & \mathcal{E}^*(X) \ar[l]_{\qquad i} \ar[d]^{\pi_p} \\
& & & C^*(\mathcal{U},\mathbb{Z}) \ar[r]^{ i\qquad} &
\bigoplus_{r+s=*}C^r(\mathcal{U},\mathcal{E}_p^s) &
\mathcal{E}^*(X,p) \ar[l]_{\qquad i} }

\begin{thm}
The morphism of spark complexes $(\pi_p, \pi_p, id): (F^*, E^*, I^*)\rightarrow(F^*_p, E^*_p, I^*_p)$
induces a surjective group homomorphism $$\Pi_p: \hat{\mathbf{H}}^*_{smooth}(X)\rightarrow
\hat{\mathbf{H}}^*_{smooth}(X,p)$$ whose kernel is an ideal. Hence, $\hat{\mathbf{H}}^*_{smooth}(X,p)$
carries a ring structure.
\end{thm}

\begin{proof}
The proof is similar to Theorem 3.6.
It's plain to see that the diagram above commutes and $\pi_p$ commutes with differentials.
Hence, the induced map $a\mapsto\pi_p(a)$ on sparks descends to a
group homomorphism $\Pi_p:\hat{\mathbf{H}}^k_{smooth}(X)\rightarrow
\hat{\mathbf{H}}^k_{smooth}(X,p)$.

To prove the surjectivity, consider a spark $a\in F^k_p$ with $D_pa=e-r$
for some $e\in E^{k+1}_p$ and $r\in I^{k+1}_p=I^{k+1}$.
We can choose some smooth form which represents same cohomology class with $r$ in
$H^{k+1}(F^*)\cong H^{k+1}(E^*)$, so there exist $a_0\in F^k$, $e_0\in E^{k+1}$ such that
$Da_0=e_0-r$. We have
$$\pi_p (Da_0)=\pi_p(e_0)-r \Rightarrow D_p(\pi_p a_0)=\pi_p e_0-r.$$
Hence, $D_p(a-\pi_p a_0)=e-\pi_pe_0$ is a smooth form.
It follows by \cite[Lemma 1.5]{HL1} that there exist
$b\in F^{k-1}_p$ and $f\in E^k_p$ with $a-\pi_pa_0=f+D_pb$.
Set $\tilde{a}=a_0+f+Db$ and note that $D\tilde{a}=Da_0+Df+DDb=e_0-r+df=(e_0+df)-r$.
Hence, $\tilde{a}$ is a spark of degree $k$ and
$\pi_p\tilde{a}=\pi_p(a_0+f+Db)=\pi_pa_0+f+D_pb=a$. So $\Pi_p$ is surjective.

We need the following lemma to show the kernel is an ideal.

\begin{lem}
On $\hat{\mathbf{H}}^k_{smooth}(X)$, one has that $\ker(\Pi_p)=\{\alpha \in
\hat{\mathbf{H}}^k_{smooth}(X) : \exists a\in \alpha \text{ where } a\in
\mathcal{E}^k(X)\subset C^0(\mathcal{U},\mathcal{E}^k)
\text{ and } \pi_p(a)=0\}$. In particular, $\ker\Pi_p\subset
\hat{\mathbf{H}}^k_{\infty}(X)$.
\end{lem}
\begin{proof}
One direction is clear. Suppose $\alpha \in \ker(\Pi_p)$ and choose
any spark $a \in \alpha$ with $Da=e-r$. $\Pi_p(\alpha)=0$ means that there exist
$b \in F^{k-1}_p$ and $s\in I^k_p= I^k$ with
$\pi_p(a)=D_pb+s=\pi_p(Db)+s$ which implies $D_p(\pi_pa)=\delta s$.
On the other hand, $D_p(\pi_pa)=\pi_p(Da)=\pi_pe-r$. So we have $\pi_pe=0$
and $-r=\delta s$.
Replace $a$ by $\bar{a}=a-Db-s$, then $\bar{a}$ represents the same class
as $a$ and $\pi_p(\bar{a})=\pi_p(a-Db)-s=0$.

In fact, we can choose $\bar{a}$ in
$\mathcal{E}^k(X)\subset C^0(\mathcal{U},\mathcal{E}^k)$.
Since $$D\bar{a}=Da-\delta s=e-(r+\delta s)=e$$
is a global smooth form, it follows by \cite[Lemma 1.5]{HL1} and the fact
$$H^*(\bigoplus_{r+s=*}C^r(\mathcal{U},F^p\mathcal{E}^s))\cong H^*(F^p\mathcal{E}^*(X))$$
that we can choose $\bar{a}$ to be smooth.
Note that $F^0\mathcal{E}^*\supset F^1\mathcal{E}^*\supset
\cdots \supset F^p\mathcal{E}^*\supset \cdots$ is the naive filtration.
\end{proof}

By the product formula of $\hat{\mathbf{H}}^*_{smooth}(X)$ \cite{H1},
it is easy to see the kernel is an ideal.
Hence, $\hat{\mathbf{H}}^*_{smooth}(X,p)$ carries a ring structure induced from
$\hat{\mathbf{H}}^*_{smooth}(X)$.

\end{proof}
\section{\v{C}ech-Dolbeault Hypersparks of Level p}
Now we introduce the \v{C}ech-Dolbeault hyperspark complex of level $p$ which set up
a bridge connecting the \v{C}ech-Dolbeault spark complex of level $p$ and the Dolbeault-Federer spark
complex of level $p$.

Fix a good cover $\mathcal{U}$ of $X$ and consider total complex of
the following double complex with total
differential $D_p=\delta+(-1)^rd_p$:

\xymatrix{
\vdots  & \vdots  &\vdots  &
 & \vdots  &\\
C^0(\mathcal{U},\mathcal{D}'^2_p) \ar[r]^{\delta}\ar[u]^{d_p}  &
C^1(\mathcal{U},\mathcal{D}'^2_p) \ar[r]^{\delta} \ar[u]^{-d_p}
& C^2(\mathcal{U},\mathcal{D}'^2_p) \ar[r]^{\delta} \ar[u]^{d_p} &
\cdots\quad\cdots \ar[r]^{\delta} &
C^r(\mathcal{U},\mathcal{D}'^2_p) \ar[u]^{(-1)^rd_p} \ar[r]^{\delta} &  \cdots   \\
C^0(\mathcal{U},\mathcal{D}'^1_p) \ar[r]^{\delta}\ar[u]^{d_p}  &
C^1(\mathcal{U},\mathcal{D}'^1_p) \ar[r]^{\delta} \ar[u]^{-d_p}
& C^2(\mathcal{U},\mathcal{D}'^1_p) \ar[r]^{\delta} \ar[u]^{d_p} &
\cdots\quad\cdots \ar[r]^{\delta} &
C^r(\mathcal{U},\mathcal{D}'^1_p) \ar[u]^{(-1)^rd_p} \ar[r]^{\delta} &  \cdots   \\
C^0(\mathcal{U},\mathcal{D}'^0_p) \ar[r]^{\delta}\ar[u]^{d_p}  &
C^1(\mathcal{U},\mathcal{D}'^0_p) \ar[r]^{\delta} \ar[u]^{-d_p}
& C^2(\mathcal{U},\mathcal{D}'^0_p) \ar[r]^{\delta} \ar[u]^{d_p} &
\cdots\quad\cdots \ar[r]^{\delta} &
C^r(\mathcal{U},\mathcal{D}'^0_p) \ar[u]^{(-1)^rd_p} \ar[r]^{\delta} &  \cdots \\
}

It is easy to see the row complexes are exact everywhere except the first column on the left, and
$$\{\ker(\delta) \text{ on the left column}\}\cong\{\text{global sections of sheaves } \mathcal{D}'^*_p\}
= \mathcal{D}'^*(X,p).$$
Hence, $$H^*(\bigoplus_{r+s=*}C^r(\mathcal{U},\mathcal{D}'^s_p))\cong H^*(\mathcal{D}'^*(X,p))
\cong H^*(\mathcal{E}^*(X,p))\cong H^*(X,p).$$

Note that every column complex is exact everywhere except at the bottom and the level of $p$ from the bottom.

Now we consider the triple of complexes
$$(F^*_p,E^*_p,I^*_p)\equiv(\bigoplus_{r+s=*}C^r(\mathcal{U},\mathcal{D}'^s_p),
\mathcal{E}^*(X,p),\bigoplus_{r+s=*}C^r(\mathcal{U},\mathcal{IF}^s)).$$
And we have
\begin{prop}
The triple of complexes $(F^*_p,E^*_p,I^*_p)$ as defined above is a spark complex,
which is called the \textbf{\v{C}ech-Dolbeault hyperspark complex of level} $\mathbf{p}$,
or more simply, the \textbf{hyperspark complex of level} $\mathbf{p}$.
\end{prop}
\begin{proof}
We have shown that $E^*_p\hookrightarrow F^*_p$ induces an isomorphism
$H^*(E^*_p)\cong H^*(F^*_p)$.
Also there is a map $$\Psi_p:I^*_p\equiv \bigoplus_{r+s=*}C^r(\mathcal{U},\mathcal{IF}^s) \hookrightarrow
\bigoplus_{r+s=*}C^r(\mathcal{U},\mathcal{D}'^s) \stackrel{\pi_p}\longrightarrow
\bigoplus_{r+s=*}C^r(\mathcal{U},\mathcal{D}'^s_p)
\equiv F^*_p.$$

And $E^k_p\cap I^k_p=\{0\}$ for $k>0$ follows \cite[Appendix B]{HL2}.
\end{proof}

\begin{defn}
A \textbf{\v{C}ech-Dolbeault hyperspark of level} $\mathbf{p}$ of degree $k$,
or \textbf{hyperspark of level} $\mathbf{p}$ is a pair
$$(a,r)\in\bigoplus_{r+s=k}C^r(\mathcal{U},\mathcal{D}'^s_p)
\oplus\bigoplus_{r+s=k+1}C^r(\mathcal{U},\mathcal{IF}^s)$$
with the spark equations $$D_pa=e-\Psi_pr \quad \text{and}\quad Dr=0$$ where $e\in\mathcal{E}_p^{k+1}(X)\subset C^0(\mathcal{U},\mathcal{D}'^{k+1}_p)$
is of bidegree $(0,k+1)$.

Two \v{C}ech-Dolbeault sparks of level $p$, $(a,r)$ and $(a',r')$ are \textbf{equivalent} if there exist $$b\in
\bigoplus_{r+s=k-1}C^r(\mathcal{U},\mathcal{D}'^s_p) \quad \text{and}\quad
s\in \bigoplus_{r+s=k}C^r(\mathcal{U},\mathcal{IF}^s)$$
satisfying $$a-a'=D_pb+s \quad\text{and}\quad r=-Ds.$$

The equivalence class determined by a \v{C}ech-Dolbeault hyperspark $(a,r)$ will be denoted by $[(a,r)]$,
and the group of \v{C}ech-Dolbeault hyperspark classes of level $p$ will be denoted by
$\hat{\mathbf{H}}^k_{hyperspark}(X,p)$.

\end{defn}

Harvey and Lawson introduced the \textbf{hyperspark complex}
$$(F^*,E^*,I^*)=(\bigoplus_{r+s=*}C^r(\mathcal{U},\mathcal{D}'^s),\mathcal{E}^*(X),
\bigoplus_{r+s=*}C^r(\mathcal{U},\mathcal{IF}^s))$$ in \cite{HL1}.
The hyperspark complex and the \v{C}ech-Dolbeault hyperspark complex of level $p$ is related by
the natural morphism $(\pi_p,\pi_p,id):
(F^*,E^*,I^*)\longrightarrow (F^*_p,E^*_p,I^*_p)$. Explicitly, we have the following
commutative diagram

 \xymatrix{& & & \bigoplus_{r+s=*}C^r(\mathcal{U},\mathcal{IF}^s)
\ar[r]^{i} \ar[d]^{id} & \bigoplus_{r+s=*}C^r(\mathcal{U},\mathcal{D}'^s)
 \ar[d]^{\pi_p} & \mathcal{E}^*(X) \ar[l]_{\qquad i} \ar[d]^{\pi_p} \\
& & & \bigoplus_{r+s=*}C^r(\mathcal{U},\mathcal{IF}^s) \ar[r]^{ i} &
\bigoplus_{r+s=*}C^r(\mathcal{U},\mathcal{D}'^s_p) &
\mathcal{E}^*(X,p) \ar[l]_{\qquad i} }

Similar to last two sections, we have the following lemma and theorem
\begin{lem}
On $\hat{\mathbf{H}}^k_{hyperspark}(X)$, one has that $\ker(\Pi_p)=\{\alpha \in
\hat{\mathbf{H}}^k_{hyperspark}(X) : \exists a\in \alpha \text{ where } a\in
\mathcal{E}^k(X)\subset C^0(\mathcal{U},\mathcal{D}'^k)
\text{ and } \pi_p(a)=0\}$. In particular, $\ker(\Pi_p)\subset
\hat{\mathbf{H}}^k_{\infty}(X)$.
\end{lem}

\begin{thm}
The morphism of spark complexes $(\pi_p, \pi_p, id): (F^*, E^*, I^*)\rightarrow(F^*_p, E^*_p, I^*_p)$
induces a surjective group homomorphism $$\Pi_p: \hat{\mathbf{H}}^*_{hyperspark}(X)\rightarrow
\hat{\mathbf{H}}^*_{hyperspark}(X,p)$$ whose kernel is an ideal. Hence, $\hat{\mathbf{H}}^*_{hyperspark}(X,p)$
carries a ring structure.
\end{thm}

Harvey and Lawson showed
\begin{thm}\cite{HL1}
Both the de Rham-Federer spark complex and the smooth hyperspark complex
are quasi-isomorphic to the hyperspark complex. Hence,
$$\hat{\mathbf{H}}^*_{spark}(X)\cong\hat{\mathbf{H}}^*_{hyperspark}(X)
\cong\hat{\mathbf{H}}^*_{smooth}(X).$$
\end{thm}

Similarly, we establish relations among the Dolbeault-Federer spark complex, the \v{C}ech-Dolbeault
spark complex and the \v{C}ech-Dolbeault hyperspark complex of level $p$.

\begin{thm}
We have morphisms of spark complexes

\xymatrix{ \{{\text{the de Rham-Federer}\atop \text{spark complex}}\} \ar[r]^{i} \ar[d]^{\pi_p}
 &  \{\text{the hyperspark complex}\}
 \ar[d]^{\pi_p} & \{\text{the smooth hyperspark complex}\} \ar[l]_{i} \ar[d]^{\pi_p} \\
 \{{\text{the Dolbeault-Federer}\atop \text{spark complex of level } p} \} \ar[r]^{ i} &
\{{\text{the \v{C}ech-Dolbeault hyperspark}\atop\text{complex of level }p}\} &
\{{\text{the \v{C}ech-Dolbeault spark}\atop\text{complex of level }p}\} \ar[l]_{i} }

where horizontal morphisms are quasi-isomorphisms.

Hence we get induced homomorphisms

\xymatrix{& & \hat{\mathbf{H}}^*_{spark}(X)
\ar[r]^{=} \ar[d]^{\Pi_p} &
\hat{\mathbf{H}}^*_{hyperspark}(X)
 \ar[d]^{\Pi_p} &
\hat{\mathbf{H}}^*_{smooth}(X)
 \ar[l]_{=} \ar[d]^{\Pi_p} \\
& & \hat{\mathbf{H}}^*_{spark}(X,p)
\ar[r]^{ =} &
\hat{\mathbf{H}}^*_{hyperspark}(X,p)
&
\hat{\mathbf{H}}^*_{smooth}(X,p)
 \ar[l]_{=} }

where the horizontal ones are isomorphism.
\end{thm}

\begin{proof}
It is easy to see we have the following two commutative diagrams

\xymatrix{ & \mathcal{IF}^*(X)
 \ar[r]^{\Psi_p} \ar[d]^{i} &\mathcal{D}'^*(X,p)
 \ar [d]^i
& \mathcal{E}^*(X,p)  \ar[l]_i \ar @{=}[d] \\
& \bigoplus_{r+s=*}C^r(\mathcal{U},\mathcal{IF}^s) \ar[r]^{\Psi_p}
& \bigoplus_{r+s=*}C^r(\mathcal{U},\mathcal{D}'^s_p)
 & \mathcal{E}^*(X,p) \ar[l]_{\qquad i}\\}

 and

\xymatrix{ & C^*(\mathcal{U},\mathbb{Z})
 \ar[r]^{i} \ar[d]^{i} &
\bigoplus_{r+s=*}C^r(\mathcal{U},\mathcal{E}_p^s) \ar [d]^i
& \mathcal{E}^*(X,p)  \ar [l]_{\qquad i} \ar @{=}[d] \\
& \bigoplus_{r+s=*}C^r(\mathcal{U},\mathcal{IF}^s) \ar[r]^{\Psi_p}
& \bigoplus_{r+s=*}C^r(\mathcal{U},\mathcal{D}'^s_p)
 & \mathcal{E}^*(X,p) \ar [l]_{\qquad i}\\}

where $$i:\mathcal{IF}^*(X)\longrightarrow\bigoplus_{r+s=*}C^r(\mathcal{U},\mathcal{IF}^s) \quad
\text{ and } \quad
i:C^*(\mathcal{U},\mathbb{Z})\longrightarrow\bigoplus_{r+s=*}C^r(\mathcal{U},\mathcal{IF}^s)$$
are quasi-isomorphisms of cochain complexes.

\end{proof}

So far, we have introduced three spark complexes associated to a complex manifold $X$, and showed
the natural isomorphisms between the groups of spark classes associated to them.
We denote the groups of spark classes by $\hat{\mathbf{H}}^*(X,p)$ collectively, and call them
the \textbf{Harvey-Lawson spark characters of level p associated to} $X$. The ring structure
of $\hat{\mathbf{H}}^*(X,p)$ is induced from the ring structure of $\hat{\mathbf{H}}^*(X)$.
There are two different ways to define the product in $\hat{\mathbf{H}}^*(X)$.
Harvey, Lawson and Zweck \cite{HLZ} defined the product via the de Rham-Federer spark complex.
And the author defined the product via the smooth hyperspark complex in \cite{H1} and
showed two ring structures are equivalent. In the next section, we shall define the product for
Deligne cohomology by both theories.

\section{The Ring Structure on Deligne Cohomology with Analytic Formula}
Deligne cohomology was invented by Deligne in 1970's. In \cite{B}, Beilinson defined the
ring structure on Deligne cohomology. In 1995, Deninger \cite{D} defined higher operations
--- Massey products in Deligne cohomology. In this section, we shall give a product formula for
Deligne cohomology via spark theory. A construction of Massey products in Deligne cohomology
will appear in \cite{H2}.

We follow \cite{EV} to define Deligne cohomology and its ring structure.

\begin{defn}
Let $X$ be a complex manifold. For $p\geq 0$, the Deligne complex
$\mathbb{Z}_{\mathcal{D}}(p)$ is the complex of sheaves:
$$0\rightarrow\mathbb{Z}\stackrel{i}\rightarrow\Omega^0\stackrel{d}\rightarrow\Omega^1
\stackrel{d}\rightarrow\cdots\stackrel{d}\rightarrow\Omega^{p-1}\rightarrow 0$$
where $\Omega^k$ denotes the sheaf of holomorphic $k$-forms on $X$.
The hypercohomology groups $\mathbb{H}^q(X,\mathbb{Z}_{\mathcal{D}}(p))$ are called
the Deligne cohomology groups of $X$, and are denoted by $H^q_{\mathcal{D}}(X,\mathbb{Z}(p))$.
\end{defn}

\begin{rmk}
In Deligne complex $\mathbb{Z}_{\mathcal{D}}(p)$, we always consider that $\mathbb{Z}$ is of degree $0$, and
$\Omega^k$ is of degree $k+1$.
\end{rmk}

\begin{ex}
It is easy to see $H^q_{\mathcal{D}}(X,\mathbb{Z}(0))=H^q(X,\mathbb{Z})$ and
$H^q_{\mathcal{D}}(X,\mathbb{Z}(1))=H^{q-1}(X,\mathcal{O}^*)$.
\end{ex}

In \cite{B}, Beilinson defined a cup product
$$\cup: \mathbb{Z}_{\mathcal{D}}(p)\otimes \mathbb{Z}_{\mathcal{D}}(p')
\rightarrow \mathbb{Z}_{\mathcal{D}}(p+p')$$ by
$$x\cup y=\left\{
    \begin{array}{ll}
      x\cdot y & \hbox{if } \deg x=0;\\
      x\wedge dy & \hbox{if } \deg x>0 \hbox{ and } \deg y=p';\\
      0 & \hbox{otherwise.}
    \end{array}
  \right.$$
The cup product $\cup$ is a morphism of complexes and associative \cite{EV} \cite{Br},
hence induces a ring structure on
$$\bigoplus_{p,q}H^q_{\mathcal{D}}(X,\mathbb{Z}(p)).$$

We are identifying the Deligne cohomology groups with subgroups of the groups of $\bar{d}$-spark classes.
Then we give a product formula for Deligne cohomology.

\begin{lem}
We have the short exact sequence $$0 \rightarrow H^k_{\mathcal{D}}(X,\mathbb{Z}(p)) \rightarrow
\hat{\mathbf{H}}^{k-1}_{spark}(X,p) \rightarrow \mathcal{Z}^k_{\mathbb{Z}}(X,p) \rightarrow 0$$
which is the middle row of $3\times 3$ diagram for the group of $\bar{d}$-spark classes of level $p$.
Hence, for any Deligne class $\alpha \in H^k_{\mathcal{D}}(X,\mathbb{Z}(p))$, there exists a spark representative
$$(a,r)\in \mathcal{D}'^{k-1}(X,p) \oplus \mathcal{IF}^k(X) \quad\text{with}\quad d_p a=-\Psi_p(r),\quad dr=0.$$
\end{lem}

\begin{proof}

By Proposition 2.8, it suffices to show
$$H^k_{\mathcal{D}}(X,\mathbb{Z}(p)) \cong
H^{k-1}(\text{Cone}(\Psi: \mathcal{IF}^*(X) \rightarrow \mathcal{D}'^*(X,p))).$$

By definition,  $H^*_{\mathcal{D}}(X,\mathbb{Z}(p))$ is the hypercohomology of
the complex of sheaves $$0\rightarrow \mathbb{Z}\rightarrow \Omega^0
\rightarrow \Omega^1 \rightarrow \ldots \rightarrow \Omega^{p-1} \rightarrow 0.$$
In other words, it is the hypercohomology of the Cone
$(\mathbb{Z} \rightarrow \Omega^{*<p})[-1]$.

Consider the acyclic resolutions:
$$\mathbb{Z}\rightarrow \mathcal{IF}^* \quad\text{and}\quad
\Omega^k\rightarrow \mathcal{D}'^{k,*}.$$

And we have quasi-isomorphism of complexes of sheaves
$$ \text{Cone}(\mathbb{Z} \rightarrow \Omega^{*<p})
\simeq\text{Cone}(\Psi: \mathcal{IF}^* \rightarrow \bigoplus_{s+t=*,s<p}\mathcal{D}'^{s,t}),$$
and hence
$$H^k_{\mathcal{D}}(X,\mathbb{Z}(p)) \cong
 \mathbb{H}^{k-1}(\text{Cone}(\mathbb{Z} \rightarrow \Omega^{*<p})) \cong
\mathbb{H}^{k-1}(\text{Cone}(\Psi: \mathcal{IF}^* \rightarrow \bigoplus_{s+t=*,s<p}\mathcal{D}'^{s,t}))$$
$$\cong H^{k-1}(\text{Cone}(\Psi: \mathcal{IF}^*(X) \rightarrow \mathcal{D}'^*(X,p))).$$

Then for any Deligne class $\alpha \in H^k_{\mathcal{D}}(X,\mathbb{Z}(p)) \subset
\hat{\mathbf{H}}^{k-1}(X,p)$, we can find a representative
$(a,r)\in \mathcal{D}'^{k-1}(X,p) \oplus \mathcal{I}^k(X)$ with $d_p a=e-\Psi_p(r)$, $dr=0$.
And we have $e=0$ since $\alpha \in \ker{\delta_1}$.

\end{proof}

Applying the representation of Deligne cohomology classes in terms of currents above, we
define a product in Deligne cohomology
$$H^k_{\mathcal{D}}(X,\mathbb{Z}(p))\otimes H^l_{\mathcal{D}}(X,\mathbb{Z}(q))
\longrightarrow H^{k+l}_{\mathcal{D}}(X,\mathbb{Z}(p+q)).$$

First, for any Deligne class $\alpha \in H^k_{\mathcal{D}}(X,\mathbb{Z}(p))$, we choose a spark representative
$$(a,r)\in \mathcal{D}'^{k-1}(X,p) \oplus \mathcal{IF}^k(X)\quad\text{with}\quad d_p a=-\Psi_p(r),\quad dr=0.$$
Similarly, for any $\beta  \in H^l_{\mathcal{D}}(X,\mathbb{Z}(q))$, we choose a spark representative
$$(b,s)\in \mathcal{D}'^{l-1}(X,q) \oplus \mathcal{IF}^l(X)\quad\text{with}\quad d_q b=-\Psi_q(s),\quad ds=0.$$

Since $\Pi_p: \hat{\mathbf{H}}^*(X) \rightarrow \hat{\mathbf{H}}^*(X,p)$ is surjective, there exist
$$(\tilde{a},r) \in \mathcal{D}'^{k-1}(X) \oplus \mathcal{IF}^k(X) \quad\text{with}\quad
\Pi_p[(\tilde{a},r)]=[(\pi_p(\tilde{a}),r)]=[(a,r)]=\alpha,$$
and
$$(\tilde{b},s) \in \mathcal{D}'^{l-1}(X) \oplus \mathcal{IF}^l(X) \quad\text{
with}\quad \Pi_q[(\tilde{b},s)]=[(\pi_q(\tilde{b}),s)]=[(b,s)]=\beta.$$

Write the spark equations for $\tilde{a}$ and $\tilde{b}$ as
$$d\tilde{a}=e-r\quad \text{and}\quad d\tilde{b}=f-s,$$ where $\pi_p\tilde{a}=a$, $\pi_pe=0$ and
$\pi_q\tilde{b}=b$, $\pi_qf=0$.

By the product formula \cite[Theorem 3.5]{HL1}
$$[(\tilde{a},r)]*[(\tilde{b},s)]=[(\tilde{a}\wedge f+(-1)^k r\wedge \tilde{b},r \wedge s)]
=[(\tilde{a} \wedge s +(-1)^ke\wedge \tilde{b},r\wedge s)].$$
Since $$d(\tilde{a}\wedge f+(-1)^k r\wedge \tilde{b})=e\wedge f-r\wedge s \quad\text{
and}\quad \pi_{p+q}(e\wedge f)=0,$$
we get $\Pi_{p+q}[(\tilde{a}\wedge f+(-1)^k r\wedge \tilde{b},
r\wedge s)] \in H^{k+l}_\mathcal{D}(X,\mathbb{Z}(p+q))$, and define it as $\alpha*\beta$.

Now we show the product is well-defined, i.e. it is independent of the choices of representatives
$\alpha$ and $\beta$. If we have another representative $(a',r')\in \alpha$ and a lift
$(\tilde{a}',r')$ with $\Pi_p[(\tilde{a}',r')]=[(a',r')]=\alpha$, then
$[(\tilde{a},r)-(\tilde{a}',r')] \in \ker \Pi_p$. By Lemma 3.7, there exists a representative
of spark class $[(\tilde{a},r)-(\tilde{a}',r')]$, which is of form $(c,0)$ where $c$ is smooth
and $\pi_p(c)=0$.
Then we have

\begin{eqnarray*}
& &\Pi_{p+q}([(\tilde{a},r)]*[(\tilde{b},s)]-[(\tilde{a}',r')]*[(\tilde{b},s)])\\
&=&\Pi_{p+q}([(\tilde{a},r)-(\tilde{a}',r')]*[(\tilde{b},s)])\\
&=&\Pi_{p+q}([(c,0)]*[(\tilde{b},s)])\\
&=&\Pi_{p+q}([(c\wedge f+(-1)^{k}0\wedge \tilde{b},0)])\\
&=&0
               \end{eqnarray*}
Similarly, we can show the product does not depend on representatives of $\beta$ either.

\begin{rmk}
In the process above, we can always choose good representatives
$a$, $\tilde{a}$, $b$, $\tilde{b}$, $r$ and $s$ in sense of \cite[Proposition 3.1]{HL1}
such that all wedge products are well defined.
\end{rmk}

\begin{thm} \emph{Product formula of Deligne cohomology I} \\
For any Deligne class $\alpha \in H^k_{\mathcal{D}}(X,\mathbb{Z}(p))$ and
$\beta  \in H^l_{\mathcal{D}}(X,\mathbb{Z}(q))$, there exist spark representations
$(a,r)$ for $\alpha$ and $(b,s)$ for $\beta$ as above. Let $(\tilde{a},r)$ and
$(\tilde{b},s)$ be de Rham-Federer sparks which are lifts of $(a,r)$ and $(b,s)$.
Then
$$\alpha * \beta=\Pi_{p+q}[(\tilde{a} \wedge s +(-1)^ke\wedge \tilde{b},r\wedge s)]=
[(\pi_{p+q}(\tilde{a} \wedge s +(-1)^ke\wedge b),r\wedge s)]\in
H^{k+l}_{\mathcal{D}}(X,\mathbb{Z}(p+q)).$$
\end{thm}
\begin{proof}
We have shown the product is well defined. In Theorem 6.11, we shall verify that
this product is equivalent to Beilinson's definition.
\end{proof}

\begin{rmk}
Suppose $X$ is a algebraic manifold and $CH^*(X)$ is the Chow ring of $X$. Considering
every nonsingular subvariety as a integrally flat current, we can define
the group homomorphism $$\psi: CH^p(X)\rightarrow H^{2p}_{\mathcal{D}}(X,\mathbb{Z}(p)).$$
By our product formula, it is quite easy to see this map induces a ring homomorphism, i.e.
the ring structure of Deligne cohomology is compatible with the ring structure of the Chow ring.
We shall explain this in the next section.
\end{rmk}

Now we rewrite last theorem in terms of the \v{C}ech-Dolbeault spark complex and
give a similar formula. Then we show this product is equivalent to the product defined by
Beilinson in Theorem 6.11.

\begin{lem}
We have the short exact sequence $$0 \rightarrow H^k_{\mathcal{D}}(X,\mathbb{Z}(p)) \rightarrow
\hat{\mathbf{H}}^{k-1}_{smooth}(X,p) \rightarrow \mathcal{Z}^k_{\mathbb{Z}}(X,p) \rightarrow 0.$$
For any Deligne class $\alpha \in H^k_{\mathcal{D}}(X,\mathbb{Z}(p))$, there exists a representative
$$(a,r)\in \bigoplus_{r+s=k-1}C^r(\mathcal{U},\mathcal{E}_p^s)\oplus C^k(\mathcal{U},\mathbb{Z})
\quad\text{with}\quad D_p a=-r,\quad \delta r=0.$$
\end{lem}
\begin{proof}
Note that we use $(a,r)$ to represent a \v{C}ech-Dolbeault spark here although we can omit $r$.
The reason is that we can make the proof of Theorem 6.11 clearer with this representation.

Applying the following quasi-isomorphisms of complexes of sheaves:
$$\mathbb{Z}\simeq\mathcal{C}^*(\mathcal{U},\mathbb{Z}) \quad \text{ and } \quad
\Omega^{*<p}\simeq\mathcal{E}^*_p\simeq\bigoplus_{r+s=*}\mathcal{C}^r(\mathcal{U},\mathcal{E}^s_p),$$
we get
$$H^k_{\mathcal{D}}(X,\mathbb{Z}(p)) \cong
 \mathbb{H}^{k-1}(\text{Cone}(\mathbb{Z} \rightarrow \Omega^{*<p}) \cong
\mathbb{H}^{k-1}(\text{Cone}(\mathcal{C}^*(\mathcal{U},\mathbb{Z}) \rightarrow \bigoplus_{r+s=*}\mathcal{C}^r(\mathcal{U},\mathcal{E}^s_p)))$$
$$\cong H^{k-1}(\text{Cone}(C^*(\mathcal{U},\mathbb{Z}) \rightarrow
\bigoplus_{r+s=*}C^r(\mathcal{U},\mathcal{E}^s_p))).$$

By Proposition 2.8 and definition of $\hat{\mathbf{H}}^{*}_{smooth}(X,p)$,
we have the short exact sequence:
$$0 \rightarrow H^k_{\mathcal{D}}(X,\mathbb{Z}(p)) \rightarrow
\hat{\mathbf{H}}^{k-1}_{smooth}(X,p) \rightarrow \mathcal{Z}^k_{\mathbb{Z}}(X,p) \rightarrow 0.$$

Hence, for any Deligne class $\alpha \in H^k_{\mathcal{D}}(X,\mathbb{Z}(p)) \subset
\hat{\mathbf{H}}^{k-1}_{smooth}(X,p)$, we can find a representative
$(a,r)\in \bigoplus_{r+s=k-1}C^r(\mathcal{U},\mathcal{E}^s_p) \oplus
C^k(\mathcal{U},\mathbb{Z})$ with $D_p a=e-r$, $\delta r=0$.
And we have $e=0$ since $\alpha \in \ker{\delta_1}$.
\end{proof}

Via the \v{C}ech-Dolbeault spark complex, we establish another product formula
for Deligne cohomology.

Our goal is to define the product in Deligne cohomology
$$H^k_{\mathcal{D}}(X,\mathbb{Z}(p))\otimes H^l_{\mathcal{D}}(X,\mathbb{Z}(q))
\longrightarrow H^{k+l}_{\mathcal{D}}(X,\mathbb{Z}(p+q)).$$

First, we choose a \v{C}ech-Dolbeault spark representative
$$(a,r)\in \bigoplus_{r+s=k-1}C^r(\mathcal{U},\mathcal{E}^s_p) \oplus
C^k(\mathcal{U},\mathbb{Z})\quad \text{with}\quad D_p a=-r,\quad\delta r=0$$
for Deligne class $\alpha \in H^k_{\mathcal{D}}(X,\mathbb{Z}(p))$,
and a \v{C}ech-Dolbeault spark representative
$$(b,s)\in \bigoplus_{r+s=l-1}C^r(\mathcal{U},\mathcal{E}^s_p) \oplus
C^l(\mathcal{U},\mathbb{Z})\quad \text{with}\quad D_q b=-s,\quad\delta s=0$$
for $\beta  \in H^l_{\mathcal{D}}(X,\mathbb{Z}(q))$.

Since $\Pi_p: \hat{\mathbf{H}}^*_{smooth}(X) \rightarrow \hat{\mathbf{H}}^*_{smooth}(X,p)$
is surjective, there exist smooth hypersparks
$$(\tilde{a},r) \in \bigoplus_{r+s=k-1}C^r(\mathcal{U},\mathcal{E}^s) \oplus
C^k(\mathcal{U},\mathbb{Z}) \quad\text{ with }\quad
\Pi_p[(\tilde{a},r)]=[(a,r)]=\alpha,$$ and
$$(\tilde{b},s) \in \bigoplus_{r+s=l-1}C^r(\mathcal{U},\mathcal{E}^s) \oplus
C^l(\mathcal{U},\mathbb{Z}) \quad\text{ with }\quad \Pi_q[(\tilde{b},s)]=[(b,s)]=\beta.$$

Write the spark equations of $\tilde{a}$ and $\tilde{b}$ as
$$D\tilde{a}=e-r\quad\text{and}\quad D\tilde{b}=f-s,$$ where $\pi_p\tilde{a}=a$, $\pi_pe=0$ and
$\pi_q\tilde{b}=b$, $\pi_qf=0$.

By the product formula of $\hat{\mathbf{H}}^{*}_{smooth}(X)$ constructed in \cite{H1},
$$[(\tilde{a},r)]*[(\tilde{b},s)]=[\tilde{a}\cup f+(-1)^k r\cup \tilde{b},r \cup s]
=[\tilde{a} \cup s +(-1)^ke\cup \tilde{b},r\cup s].$$
We have $$D(\tilde{a}\cup f+(-1)^k r\cup \tilde{b})=e\wedge f-r\cup s \text{
and } \pi_{p+q}(e\wedge f)=0,$$ so we get $\Pi_{p+q}[(\tilde{a}\cup f+(-1)^k r\cup \tilde{b},
r\cup s)] \in H^{k+l}_\mathcal{D}(X,\mathbb{Z}(p+q))$, which is defined to be $\alpha*\beta$.

The product is only dependent on the spark classes
$\alpha$ and $\beta$. If we have another representative $(a',r')\in \alpha$ and a lift
$(\tilde{a}',r')$ with $\Pi_p[(\tilde{a}',r')]=[(a',r')]=\alpha$, then
$[(\tilde{a},r)-(\tilde{a}',r')] \in \ker \Pi_p$. By Lemma 4.4, we can choose a representative
of spark class $[(\tilde{a},r)-(\tilde{a}',r')]$, which is of form $(c,0)$ where $c$ is smooth
and $\pi_p(c)=0$.
Then we have
\begin{eqnarray*}
& &\Pi_{p+q}([(\tilde{a},r)]*[(\tilde{b},s)]-[(\tilde{a}',r')]*[(\tilde{b},s)])\\
&=&\Pi_{p+q}([(\tilde{a},r)-(\tilde{a}',r')]*[(\tilde{b},s)])\\
&=&\Pi_{p+q}([(c,0)]*[(\tilde{b},s)])\\
&=&\Pi_{p+q}([(c\cup f+(-1)^{k}0\cup \tilde{b},0)])\\
&=&0
\end{eqnarray*}
Similarly, we can show the product does not depend on representatives of $\beta$ either.

\begin{thm} \emph{Product formula of Deligne cohomology II} \\
For any Deligne class $\alpha \in H^k_{\mathcal{D}}(X,\mathbb{Z}(p))$ and
$\beta  \in H^l_{\mathcal{D}}(X,\mathbb{Z}(q))$, there exist \v{C}ech-Dolbeault spark representations
$(a,r)$ for $\alpha$ and $(b,s)$ for $\beta$ as above. Let $(\tilde{a},r)$ and
$(\tilde{b},s)$ be smooth hypersparks which are lifts of $(a,r)$ and $(b,s)$.
Then
$$\alpha * \beta=\Pi_{p+q}[(\tilde{a} \cup s +(-1)^ke\cup \tilde{b},r\cup s)]=
[(\pi_{p+q}(\tilde{a} \cup s +(-1)^ke\cup b),r\cup s)]\in
H^{k+l}_{\mathcal{D}}(X,\mathbb{Z}(p+q)).$$
\end{thm}

\begin{thm}
Two product formulas in Theorem 6.5 and Theorem 6.9 are equivalent.
\end{thm}

\begin{proof}
The product formula in Theorem 6.5 and 6.9 are based on product formulas of
$\hat{\mathbf{H}}^{*}_{spark}(X,p)$ and $\hat{\mathbf{H}}^{*}_{smooth}(X,p)$ established
in \cite{HL1} and \cite{H1} respectively.
Note that $\hat{\mathbf{H}}^{*}_{spark}(X,p)\cong\hat{\mathbf{H}}^{*}_{smooth}(X,p)$ and
the ring structures on them are compatible \cite[Theorem 4.5]{H1}. Hence,
product formulas in Theorem 6.5 and 6.9 are equivalent as well.
\end{proof}

Our product formula is quite explicit compared with the product in \cite{B} which is defined
on the sheaf level. Now we verify that these products are equivalent.

\begin{thm}
The products Theorem 6.5 and 6.9 are equivalent to Beilinson's product.
\end{thm}

\begin{proof}
It is sufficient to show that the product formula in Theorem 6.9 is the same as
Beilinson's product which is induced from the cup product on the sheaf level.

The outline of the proof is following: First, we construct an explicit isomorphism
between $H^k_{\mathcal{D}}(X,\mathbb{Z}(p))$ and
$\ker \delta_1: \hat{\mathbf{H}}^{k-1}_{smooth}(X,p)\rightarrow \mathcal{Z}_{\mathbb{Z}}^{k}(X,p)$;
Then, we calculate the product induced by
$$\cup: \mathbb{Z}_{\mathcal{D}}(p)\otimes \mathbb{Z}_{\mathcal{D}}(q)
\rightarrow \mathbb{Z}_{\mathcal{D}}(p+q)$$ using \v{C}ech resolution; Finally,
we calculate the product via smooth hypersparks defined earlier in this section,
and compare these two products.

Step 1: Fix a good cover $\{\mathcal{U}\}$ of $X$ and
take \v{C}ech resolution for the complex of sheaves
$\mathbb{Z}_{\mathcal{D}}(p)\longrightarrow
\mathcal{C}^*(\mathcal{U},\mathbb{Z}_{\mathcal{D}}(p))$.

Then
$$H^q_{\mathcal{D}}(X,\mathbb{Z}(p))\equiv \mathbb{H}^q(\mathbb{Z}_{\mathcal{D}}(p))
\cong \mathbb{H}^q(Tot(\mathcal{C}^*(\mathcal{U},\mathbb{Z}_{\mathcal{D}}(p))))
\cong H^q(Tot(C^*(\mathcal{U},\mathbb{Z}_{\mathcal{D}}(p))))$$
where $C^*(\mathcal{U},\mathbb{Z}_{\mathcal{D}}(p))$ are the groups of global sections of
sheaves $\mathcal{C}^*(\mathcal{U},\mathbb{Z}_{\mathcal{D}}(p))$ and look like the
following double complex.

\xymatrix{
\vdots & \vdots  & \vdots  & \vdots  &  & \vdots    \\
C^k(\mathcal{U},\mathbb{Z}) \ar[r]^{(-1)^{k}i} \ar[u]^{\delta}
& C^k(\mathcal{U},\Omega^0) \ar[r]^{(-1)^{k}\partial} \ar[u]^{\delta}
& C^k(\mathcal{U},\Omega^1) \ar[r]^{(-1)^{k}\partial} \ar[u]^{\delta} &
C^k(\mathcal{U},\Omega^2)  \ar[r]^{\quad(-1)^{k}\partial} \ar[u]^{\delta} & \cdots \ar[r]^{(-1)^{k}\partial\qquad} &
C^k(\mathcal{U},\Omega^{p-1})  \ar[u]^{\delta}   \\
\vdots  \ar[u]^{\delta}
& \vdots  \ar[u]^{\delta}
& \vdots \ar[u]^{\delta} &
\vdots \ar[u]^{\delta} &  &
\vdots \ar[u]^{\delta}   \\
C^2(\mathcal{U},\mathbb{Z}) \ar[r]^i \ar[u]^{\delta}
& C^2(\mathcal{U},\Omega^0) \ar[r]^{\partial} \ar[u]^{\delta}
& C^2(\mathcal{U},\Omega^1) \ar[r]^{\partial} \ar[u]^{\delta} &
C^2(\mathcal{U},\Omega^2)  \ar[r]^{\partial} \ar[u]^{\delta} & \cdots \ar[r]^{\partial\qquad} &
C^2(\mathcal{U},\Omega^{p-1})  \ar[u]^{\delta}   \\
C^1(\mathcal{U},\mathbb{Z}) \ar[r]^{-i} \ar[u]^{\delta} & C^1(\mathcal{U},\Omega^0) \ar[r]^{-\partial}
\ar[u]^{\delta} & C^1(\mathcal{U},\Omega^1) \ar[r]^{-\partial} \ar[u]^{\delta} &
C^1(\mathcal{U},\Omega^2) \ar[r]^{-\partial} \ar[u]^{\delta} & \cdots \ar[r]^{-\partial\qquad} &
C^1(\mathcal{U},\Omega^{p-1})  \ar[u]^{\delta}   \\
C^0(\mathcal{U},\mathbb{Z}) \ar[r]^i \ar[u]^{\delta} & C^0(\mathcal{U},\Omega^0) \ar[r]^{\partial}
\ar[u]^{\delta} & C^0(\mathcal{U},\Omega^1) \ar[r]^{\partial} \ar[u]^{\delta} &
C^0(\mathcal{U},\Omega^2) \ar[r]^{\partial} \ar[u]^{\delta} & \cdots \ar[r]^{\partial\qquad} &
C^0(\mathcal{U},\Omega^{p-1})  \ar[u]^{\delta} \\}

Let $M^*_p \equiv Tot(C^*(\mathcal{U},\mathbb{Z}_{\mathcal{D}}(p)))$
denote the total complex of the double complex
$C^*(\mathcal{U},\mathbb{Z}_{\mathcal{D}}(p))$ with differential
$$\Delta_p(a)=\left\{
                     \begin{array}{ll}
                       (\delta+(-1)^ri)(a), & \hbox{when}\quad a\in C^r(\mathcal{U},\mathbb{Z});\\
                       (\delta+(-1)^r\partial)(a), & \hbox{when}\quad a\in C^r(\mathcal{U},\Omega^j),\quad j<p-1;\\
                       \delta a, & \hbox{when}\quad a\in C^r(\mathcal{U},\Omega^{p-1}).
                     \end{array}
                   \right.
$$

Now we construct a map
$$\varphi_p: H^*(M^*_p)\cong H^*_{\mathcal{D}}(X,\mathbb{Z}(p))\longrightarrow
\ker\delta_1\subset \hat{\mathbf{H}}^{*-1}_{smooth}(X,p).$$

Assume that a cycle $\tilde{a}\in M^k_p$ represents a Deligne class in $H^k_{\mathcal{D}}(X,\mathbb{Z}(p))$,
and $$\tilde{a}=r+a=r+\sum_{i+j=k-1, j<p}a^{i,j}$$ where
$r\in C^k(\mathcal{U},\mathbb{Z})$ and $a^{i,j}\in C^i(\mathcal{U},\Omega^{j})$.

Note that $a^{i,j}\in C^i(\mathcal{U},\Omega^{j})\subset C^i(\mathcal{U},\mathcal{E}^{j,0})
\subset C^i(\mathcal{U},\mathcal{E}^{j}_p)$,
so $a\in  \bigoplus_{i+j=k-1} C^i(\mathcal{U},\mathcal{E}^{j}_p)$.
And it is easy to see
$$\Delta_p\tilde{a}=0 \Leftrightarrow D_pa+(-1)^kr=0 \text{ and } \delta r=0,$$
where $D_p$ is the differential of the total complex of double complex
$\bigoplus_{r+s=*}C^r(\mathcal{U},\mathcal{E}^s_p)$ defined in Section 4.
Hence, $\varphi_p:\tilde{a}\rightarrow (a,(-1)^kr)$ gives a map from cycles to smooth hypersparks.
Moreover, assume $\tilde{a}$ and $\tilde{a}'$ represent the same Deligne class, i.e.
$\tilde{a}-\tilde{a}'=\Delta_p\tilde{b}$ is a boundary, where
$\tilde{b}=s+\Sigma_{i+j=k-2,j<p}b^{i,j}$ for
$s\in C^{k-1}(\mathcal{U},\mathbb{Z})$ and $b^{i,j}\in C^i(\mathcal{U},\Omega^{j})$.
Then $$a-a'+r-r'=\tilde{a}-\tilde{a}'=\Delta_p\tilde{b}=\delta s+(-1)^{k-1}i(s)+D_pb$$
implies $$a-a'=(-1)^{k-1}s+D_pb \text{ and } (-1)^kr-(-1)^kr'=-(-1)^{k-1}\delta s,$$
i.e. $(a,(-1)^kr)$ and $(a',(-1)^kr')$ represent the same spark class.
So the map ( also denoted by $\varphi_p$ )
$$\varphi_p: H^k_{\mathcal{D}}(X,\mathbb{Z}(p))\rightarrow
 \hat{\mathbf{H}}^{k-1}_{smooth}(X,p)$$
which maps a Deligne class $[\tilde{a}]$ to a smooth hyperspark class $[(a,(-1)^kr)]$ is well-defined.
$\varphi_p([\tilde{a}])=[(a,(-1)^kr)]$ satisfies the spark equation
$D_pa+(-1)^kr=0$, so we have $Im \varphi_p\subset \ker\delta_1$.
We have known $\varphi_p: H^k_{\mathcal{D}}(X,\mathbb{Z}(p))\rightarrow
 \ker\delta_1$ is an isomorphism from Lemma 6.8.

Step 2: The product formula for Deligne classes is induced by the cup product
$$\cup: \mathbb{Z}_{\mathcal{D}}(p)\otimes \mathbb{Z}_{\mathcal{D}}(q)
\rightarrow \mathbb{Z}_{\mathcal{D}}(p+q)$$
with the formula
$$x\cup y=\left\{
    \begin{array}{ll}
      x\cdot y & \hbox{if } \deg x=0;\\
      x\wedge dy & \hbox{if } \deg x>0 \hbox{ and } \deg y=q;\\
      0 & \hbox{otherwise.}
    \end{array}
  \right.$$

In the appendix of \cite{H1}, we showed the explicit product formula on \v{C}ech cycles.

Assume $$\alpha\in H^k_{\mathcal{D}}(X,\mathbb{Z}(p))\quad\text{and}\quad
\beta\in H^l_{\mathcal{D}}(X,\mathbb{Z}(q)),$$ and let
$$\tilde{a}=r+a=r+\sum_{i+j=k-1,j<p}a^{i,j}\in M^k_p \text{ be a representative of }\alpha$$
and  $$\tilde{b}=s+b=s+\sum_{i+j=l-1,i<q} b^{i,j}\in M^l_q\text{ be a representative of }\beta$$
where
$$r\in C^k(\mathcal{U},\mathbb{Z}),\quad a^{i,j}\in C^i(\mathcal{U},\Omega^{j}),$$
and
$$s\in C^l(\mathcal{U},\mathbb{Z}),\quad b^{i,j}\in C^i(\mathcal{U},\Omega^{j}).$$

By \cite[Theorem 7.1]{H1}, we calculate
\begin{eqnarray*}
\alpha\cup\beta &=&[\tilde{a}\cup \tilde{b}]\\
&=&[rs+\Sigma_{i+j=l-1,j<q}(-1)^{0\cdot i}r\cdot b^{i,j}+
\Sigma_{i+j=k-1,j<p}(-1)^{j\cdot(l-q)}a^{i,j}\wedge\partial b^{l-q,q-1}]\\
&=&[r\cup \tilde{b}+a \cup\partial b^{l-q,q-1}]\\
\end{eqnarray*}

Step 3: Let us calculate the product of $\alpha$ and $\beta$ by the formula in Theorem 6.9.

$\varphi_p(\tilde{a})=(a,(-1)^kr)$ and $\varphi_q(\tilde{b})=(b,(-1)^ls)$
are two smooth hypersparks which represent Deligne classes $\alpha$ and $\beta$ respectively.
The spark equations associated to them are:
$$D_pa=0-(-1)^kr,\quad \delta (-1)^kr=0$$
and
$$D_pb=0-(-1)^ls,\quad \delta (-1)^ls=0.$$

Because of the surjectivity of the map
$\Pi_p: \hat{\mathbf{H}}^*_{smooth}(X) \rightarrow \hat{\mathbf{H}}^*_{smooth}(X,p)$
, there exist
$$(A,(-1)^kr) \in \bigoplus_{i+j=k-1}C^i(\mathcal{U},\mathcal{E}^j) \oplus
C^k(\mathcal{U},\mathbb{Z}) \text{ with }
\Pi_p[(A,(-1)^kr)]=[(a,(-1)^kr)]=\alpha,$$ and
$$(B,(-1)^ls) \in \bigoplus_{i+j=l-1}C^i(\mathcal{U},\mathcal{E}^j) \oplus
C^l(\mathcal{U},\mathbb{Z}) \text{ with } \Pi_q[(B,(-1)^ls)]=[(b,(-1)^ls)]=\beta.$$

Assume the spark equations for $A$ and $B$ are
$DA=e-(-1)^kr$, $DB=f-(-1)^ls$, then $\pi_pA=a$, $\pi_pe=0$ and
$\pi_qB=b$, $\pi_qf=0$.

By the product formula in \cite{H1},
$$[(A,r)]*[(B,s)]=[A\cup f+(-1)^k (-1)^kr\cup B,(-1)^kr \cup (-1)^ls]
=[A\cup f+ r\cup B,(-1)^{k+l}r \cup  s].$$
We have $$D(A\cup f+ r\cup B)=e\wedge f-(-1)^{k+l}r\cup s \text{
and } \pi_{p+q}(e\wedge f)=0,$$ so we get $\Pi_{p+q}[(A\cup f+ r\cup B,
(-1)^{k+l}r\cup s)] \in H^{k+l}_\mathcal{D}(X,\mathbb{Z}(p+q))$, and define it as $\alpha*\beta$.

We compare two results under isomorphism $\varphi_{p+q}: H^k_{\mathcal{D}}(X,\mathbb{Z}(p+q))\rightarrow
 \ker\delta_1$.

The following Lemma shows that
$\varphi_{p+q}(r\cup \tilde{b}+a \cup\partial b^{l-q,q-1})=(r\cup b+a \cup\partial b^{l-q,q-1},(-1)^{k+l}r\cup s)$
and $(\pi_{p+q}(A\cup f+ r\cup B),
(-1)^{k+l}r\cup s)$ represent the same class.
\end{proof}
\begin{lem}
$r\cup b+a \cup\partial b^{l-q,q-1}=\pi_{p+q}(A\cup f+ r\cup B)+(-1)^kD_{p+q}(a\cup(B-b)).$
\end{lem}
\begin{proof}
Compare $$Db=\partial b^{l-q,q-1}+D_pb=\partial b^{l-q,q-1}-(-1)^ls$$
and $$DB=f-(-1)^ls,$$ we have
$$ \partial b^{l-q,q-1}=f-D(B-b).$$

\begin{eqnarray*}
& &\text{Right hand side}\\
&=&\pi_{p+q}(A\cup f+ r\cup B)+(-1)^kD_{p+q}(a\cup(B-b))\\
&=&\pi_{p+q}(A\cup f)+ \pi_{p+q}(r\cup B)+(-1)^k\pi_{p+q}D(a\cup(B-b))\\
&=&\pi_{p+q}(a\cup f)+ \pi_{p+q}(r\cup B)+(-1)^k\pi_{p+q}(Da\cup(B-b)+(-1)^{k-1}a\cup D(B-b))\\
&=&\pi_{p+q}(a\cup f)+ \pi_{p+q}(r\cup B)+(-1)^k\pi_{p+q}(D_pa\cup(B-b))-\pi_{p+q}(a\cup D(B-b))\\
&=&\pi_{p+q}(a\cup f)+ \pi_{p+q}(r\cup B)+(-1)^k\pi_{p+q}(-(-1)^kr\cup(B-b))-\pi_{p+q}(a\cup D(B-b))\\
&=&\pi_{p+q}(a\cup f)+ \pi_{p+q}(r\cup B)-\pi_{p+q}(r(B-b))-\pi_{p+q}(a\cup D(B-b))\\
&=&\pi_{p+q}(a\cup f-a\cup D(B-b))+ \pi_{p+q}(r\cup b) \\
&=&\pi_{p+q}(a \cup\partial b^{l-q,q-1})+ \pi_{p+q}(r\cup b) \\
&=&a \cup\partial b^{l-q,q-1}+ r\cup b \\
&=& \text{Left hand side}.
\end{eqnarray*}

\end{proof}

\section{Application to Algebraic Cycles}
We begin this section by observing that, from the viewpoint of spark theory,
it is trivial that every analytic subvariety of a complex manifold represents
a Deligne cohomology class. Furthermore, when two cycles intersect properly,
their intersection represents the product of the Deligne classes they represent.
We shall then give a proof of the rational invariance of these Deligne classes in the
algebraic setting, thereby giving the well known ring homomorphism
$$\psi: CH^*(X)\rightarrow H^{2*}_{\mathcal{D}}(X,\mathbb{Z}(*)).$$

Let $V$ be a subvariety of complex manifold $X$ with codimension $p$.
Then integration over the regular part of $V$
$$[V](\alpha)\equiv\int_{\text{Reg }V}\alpha, \quad \forall \text{ smooth form }
\alpha \text{ with compact support }$$
defines a degree $(p,p)$ current $[V]$ on $X$.
Moreover, $[V]$ is rectifiable \cite{Har}, hence $[V]\in \mathcal{IF}^{2p}(X)$.
It is easy to see $V$ represents a Deligne class.

\begin{prop}
$(0,[V])$ represents a spark class in $\hat{\mathbf{H}}^{2p-1}(X,p)$. Moreover,
this class belongs to $H^{2p}_{\mathcal{D}}(X,\mathbb{Z}(p))=\ker \delta_1\subset
\hat{\mathbf{H}}^{2p-1}(X,p)$.
\end{prop}
\begin{proof}
Since $[V]$ is of type $(p,p)$, we have $\Psi_p([V])=0$ and $(0,[V])$ satisfies
the spark equation $d0=0-\Psi_p([V])$.
\end{proof}

\begin{prop}
Let $V$, $W$ be two subvarieties which intersect properly. Then
$$[(0,[V])]*[(0,[W])]=[(0,[V\cap W])].$$
\end{prop}
\begin{proof}
Let $V$, $W$ be two subvarieties in $X$ with codimension $p$ and $q$ respectively.
Let $r$ and $s$ denote currents $[V]$ and $[W]$, then $r\wedge s=[V\cap W]$.
Now we calculate the product of two Deligne classes $[(0,r)]$ and $[(0,s)]$.
First, fix a lift of $(0,r)$, say $(a,r)$ with spark equation $da=e-r$ and
a lift of $(0,s)$, $(b,s)$ with $db=f-s$. Note that $\pi_p(a)=0$, $\pi_p(e)=0$
and $\pi_q(b)=0$, $\pi_q(f)=0$.
By product formula, $[(a,r)][(b,s)]=[(a\wedge f+r\wedge b, r\wedge s)]$.
Since $\pi_{p+q}(a\wedge f+r\wedge b)=0$, we have
$$[(0,[V])]*[(0,[W])]=\Pi_{p+q}([(a\wedge f+r\wedge b, r\wedge s)])=[(0,r\wedge s)]=[(0,[V\cap W])].$$
\end{proof}

\begin{prop}
If $X$ is an algebraic manifold of dimension $n$ and $V$ is an algebraic cycle which is rationally
equivalent to zero, then $V$ represents zero Deligne class.
\end{prop}
\begin{proof}
Assume $V$ is an algebraic cycle with dimension $k$ and codimension $p$.
If $V$ is rationally equivalent to zero, in particular, $V$ represents
zero homology class, then $V=dS$ for some rectifiable current with degree $2p-1$
( and real dimension $2k+1$ ).
Hence $(0,V)$ is equivalent to $(\pi_p(S),0)$ as sparks of level $p$.
$(\pi_p(S),0)$ represents zero class if and only of $$\pi_p(S)=d_pA+\Psi_pR
\text{ for some current } A\in \mathcal{D}'^{2p-2}(X,p) \text{ and closed current }
R\in \mathcal{IF}^{2p-1}(X),$$ i.e.
$$[\pi_p(S)]=0\in H^{2p-1}(X,\mathbb{C})/F^{p}H^{2p-1}(X,\mathbb{C})\oplus H^{2p-1}(X,\mathbb{Z})\equiv\mathcal{J}^{p},$$
which means the Abel-Jacobi invariant of $V$ is zero.
It is well known that the Abel-Jacobi invariant is trivial for a cycle rationally
equivalent to zero. So we are done. We give a short and direct proof of this fact now.

If $V$ is rationally equivalent to zero, then there is a cycle $W\subset \mathbb{P}^1\times
X$ of codimension $p$, such that $V=\pi^{-1}(1)-\pi^{-1}(0)$ where $\pi:W\rightarrow \mathbb{P}^1$,
the restriction of the projection $pr_1:\mathbb{P}^1\times X\rightarrow \mathbb{P}^1$,
is equidimensional over $\mathbb{P}^1$.
Define $V_z=\pi^{-1}(z)-\pi^{-1}(0)$, then we have a map $\mu: \mathbb{P}^1\rightarrow \mathcal{J}^{p}$
which assigns $z$ the Abel-Jacobi invariant of $V_z$. We shall show that $\mu$ is holomorphic, hence a
constant map to zero.

Let us recall the construction of the Abel-Jacobi map briefly. If $V$ is a cycle homologous
to zero, then $V=dS$. Integrating over $S$, $\int_S$ defines a class in
$$H^{2p-1}(X,\mathbb{C})/F^{p}H^{2p-1}(X,\mathbb{C})=F^{n-p+1}H^{2n-2p+1}(X,\mathbb{C})^*.$$
If $dS'=V$, then the difference $\int_S-\int_{S'}$ lies in the image of map
$$H_{2n-2p+1}(X,\mathbb{Z})\rightarrow F^{n-p+1}H^{2n-2p+1}(X,\mathbb{C})^*.$$
Therefore, we get the Abel-Jacobi invariant of $V$ defined in
$$\mathcal{J}^{p}\equiv H^{2p-1}(X,\mathbb{C})/F^{p}H^{2p-1}(X,\mathbb{C})\oplus H^{2p-1}(X,\mathbb{Z}).$$

Now we focus on the map $\mu$. Let $\gamma_z$ be a curve on $\mathbb{P}^1$ connecting
$0$ and $z$ and $S_z=\pi^{-1}(\gamma_z)$ with $dS_z=V_z$. We want to show that
$\mu: z\mapsto \int_{S_z}$ is holomorphic. Note that
$$F^{n-p+1}H^{2n-2p+1}(X,\mathbb{C})\cong\bigoplus_{r+s=2n-2p+1 \atop r\geq n-p+1}\mathcal{H}^{r,s}(X)$$
where $\mathcal{H}^{r,s}(X)$ is the group of harmonic $(r,s)$ forms.
So it suffices to show $\mu_{\alpha}:z\mapsto\int_{S_z}\alpha$ is holomorphic for every
$\alpha\in\mathcal{H}^{r,s}(X)$, $r+s=2n-2p+1$, $r\geq n-p+1$.

Let $\nu$ be a vector field in a small neighborhood $U$ of $z$ in $\mathbb{P}^1$, and
$\tilde{\nu}$ be a lift of $\nu$ in $U\times X$. If $\nu$ is of type $(0,1)$,
we have
$$\nu \int_{S_z}\alpha=
\int_{\pi^{-1}(z)}\tilde{\nu}\lrcorner \alpha=0$$
for any $\alpha\in\mathcal{H}^{r,s}(X)$, $r+s=2n-2p+1$, $r\geq n-p+1$.
The last equality follows from the fact $\tilde{\nu}\lrcorner \alpha$ has
no component of type $(n-p,n-p)$.

\end{proof}
By the last propositions and Chow's moving lemma, it is easy to see
\begin{thm}
The map $V\mapsto[(0,[V])]$ induces a ring homomorphism
$$\psi: CH^*(X)\rightarrow H^{2*}_{\mathcal{D}}(X,\mathbb{Z}(*)).$$
\end{thm}

\section{Chern Classes for Holomorphic Bundles in Deligne Cohomology}

In this section we shall construct Chern classes in Deligne cohomology
for holomorphic bundles $E$ over a complex manifold $X$. These classes have the usual properties
and map to the integral Chern classes under the ring homomorphism $H^{2*}_{\mathcal{D}}(X,\mathbb{Z}(*))
\rightarrow H^{2*}(X,\mathbb{Z})$.

In their fundamental paper \cite{CS}, Cheeger and Simons showed that for a smooth complex
vector bundle $E \rightarrow X$ with unitary connection $\nabla$, there exist refined Chern
classes $\hat{c}_k(E,\nabla) \in \hat{\mathbf{H}}^{2k-1}(X)$ with
$$\delta_1(\hat{c}_k(E,\nabla))=c_k(\Omega^{\nabla})\quad\text{and}\quad
\delta_2(\hat{c}_k(E,\nabla))=c_k(E)$$
where $c_k(E)$ is the $k$th integral Chern class and $c_k(\Omega^{\nabla})$
is the Chern-Weil form representing $c_k(E)\otimes \mathbb{R}$ in the curvature of
$\nabla$. Setting $\hat{c}(E)=1+\hat{c}_1+\hat{c}_2+...$, they showed
$$\hat{c}(E\oplus E',\nabla \oplus \nabla ')=\hat{c}(E,\nabla)*\hat{c}(E',\nabla ').$$

When $X$ is a complex manifold, we can take the projections
$$\hat{d}_k(E,\nabla) \equiv
\Pi_k(\hat{c}_k(E,\nabla)) \in \hat{\mathbf{H}}^{2k-1}(X,k).$$
By equations above and Proposition 3.8, we have
$$\delta_1(\hat{d}_k(E,\nabla))=\pi_k(c_k(\Omega^{\nabla}))\quad\text{and}\quad
\delta_2(\hat{d}_k(E,\nabla))=c_k(E).$$

Suppose that $E$ is holomorphic and is provided with a hermitian metric $h$.
Let $\nabla$ be the associated canonical hermitian connection. Then
$c_k(\Omega^{\nabla})$ is of type $k,k$ and we have
$$\delta_1(\hat{d}_k(E,\nabla))=\pi_k(c_k(\Omega^{\nabla}))=0 \quad\Longrightarrow\quad
\hat{d}_k(E,\nabla)\in \ker(\delta_1)=H^{2k}_{\mathcal{D}}(X,\mathbb{Z}(k)).$$

\begin{prop}
The class $\hat{d}_k(E,\nabla)\in H^{2k}_{\mathcal{D}}(X,\mathbb{Z}(k))$ defined above
is independent of the choice of hermitian metric.
\end{prop}

\begin{proof}
Let $h_0,h_1$ be hermitian metrics on $E$ with canonical connections
$\nabla_0,\nabla_1$ respectively. Then
$$\hat{c}_k(E,\nabla_1)-\hat{c}_k(E,\nabla_0)=[T]$$
where $[T]$ is the differential character represented by the smooth transgression form
$$T=T(\nabla_1,\nabla_0) \equiv k \int_0^1 C_k(\nabla_1-\nabla_0,\Omega_t,...,\Omega_t)dt$$
where $C_k(X_1,...,X_k)$ is the polarization of the $k$th elementary symmetric function
and where $\Omega_t$ is the curvature of the connection
$\nabla_t \equiv t\nabla_1+(1-t)\nabla_0$. Fix a local holomorphic frame field for $E$
and let $H_i$ be the hermitian matrix representing the metric $h_i$ with respect to
this trivialization. Then
$$\nabla_1-\nabla_0=\theta_1-\theta_0 \text{  where  }
\theta_j \equiv \partial H_j\cdot H_j^{-1}.$$
In this framing, $\nabla_t=d+\theta_t$ where $\theta_t=t\theta_1+(1-t)\theta_0$
and so its curvature $\Omega_t=d\theta_t-\theta_t\wedge\theta_t$ only has Hodge components
of type $1,1$ and $2,0$. It follows that the Hodge components
$$T^{p,q}=0 \text{  for  } p<q.$$
So we have $\hat{d}_k(E,\nabla_1)-\hat{d}_k(E,\nabla_0)=\Pi_k([T])=0$.

\end{proof}

\begin{rmk}
In the proof of last Proposition, it is easy to see that we can choose any connection compatible
to the complex structure ( $\nabla^{0,1}=\bar{\partial}$ ) to define the Chern classes in Deligne cohomology.
\end{rmk}

By the proposition above, each holomorphic vector bundle of rank $r$ has a well defined
total Chern class in Deligne cohomology
$$ \hat{d}(E)=1+\hat{d}_1(E)+...+\hat{d}_r(E) \in \bigoplus_{j=0}^r H_{\mathcal{D}}^{2j}(X,\mathbb{Z}(j)).$$

Denote by $\mathcal{V}^k(X)$ the set of isomorphism classes of holomorphic vector bundles of rank
$k$ on $X$, and by $\mathcal{V}(X)=\coprod_{k\geq 0}\mathcal{V}^k(X)$ the additive
monoid under Whitney sum.

\begin{thm}
On any complex manifold there is a natural transformation of functors
$$\hat{d}:\mathcal{V}(X) \rightarrow \bigoplus_j H_{\mathcal{D}}^{2j}(X,\mathbb{Z}(j))$$
with the property that:
\begin{enumerate}
  \item $\hat{d}(E\oplus F)=\hat{d}(E)*\hat{d}(F)$,
  \item $\hat{d}:\mathcal{V}^1(X) \rightarrow 1+H_{\mathcal{D}}^{2}(X,1)$ is an isomorphism,
  \item under the natural map $\kappa: H_{\mathcal{D}}^{2j}(X,\mathbb{Z}(j)) \rightarrow H^{2j}(X,\mathbb{Z})$,
  $\kappa\circ \hat{d}=c$ (the total integral Chern class).

\end{enumerate}
\end{thm}

\begin{proof}
(1)Suppose $E$ and $F$ are holomorphic bundles with hermitian connections
$\nabla$ and $\nabla'$, then we have $\hat{d}(E)=1+\hat{d}_1(E)+\hat{d}_2(E)+\cdots=
1+\Pi_1(\hat{c}_1(E,\nabla))+\Pi_2(\hat{c}_2(E,\nabla))+\cdots$ and similarly
$\hat{d}(F)=1+\Pi_1(\hat{c}_1(F,\nabla'))+\Pi_2(\hat{c}_2(F,\nabla'))+\cdots$.

Since $$\hat{c}(E\oplus F,\nabla \oplus \nabla ')=\hat{c}(E,\nabla)*\hat{c}(F,\nabla '),$$
we have
\begin{eqnarray*}
\hat{d}_k(E\oplus F)&=&\Pi_k(\hat{c}_k(E\oplus F,\nabla \oplus \nabla '))\\
&=&\Pi_k(\sum_{i=0}^{k}\hat{c}_i(E,\nabla)\cdot\hat{c}_{k-i}(F,\nabla'))\\
&=&\sum_{i=0}^{k}\Pi_i(\hat{c}_i(E,\nabla))\cdot\Pi_{k-i}(\hat{c}_{k-i}(F,\nabla'))\\
&=&\sum_{i=0}^{k}\hat{d}_i(E)\cdot\hat{d}_{k-i}(F).
\end{eqnarray*}
It is easy to see the second to last equality from our definition of product of Deligne cohomology classes.
Recall when we defined the product of two Deligne classes, we first lifted them to
two sparks, then did multiplication and projected the product back.

(2) is true because $H_{\mathcal{D}}^{2}(X,\mathbb{Z}(1))\cong H^1(X,\mathcal{O}^*)$.

(3) follows $\delta_2(\hat{d}_k(E,\nabla))=c_k(E)$.
\end{proof}

Following Grothendieck we define the holomorphic $K$-theory of $X$ to be the
quotient $$K_{\text{hol}}(X)\equiv \mathcal{V}(X)^{+}/\sim$$
where $\sim$ is the equivalence relation generated by setting
$[E]\sim[E'\oplus E'']$ when there exists a short exact sequence of
holomorphic bundles $0\rightarrow E' \rightarrow E \rightarrow E''\rightarrow 0$.
The next theorem tells us the natural transformation $\hat{d}$ defined above
descends to a natural transformation
$$\hat{d}:K_{\text{hol}}(X)\rightarrow\bigoplus_j H_{\mathcal{D}}^{2j}(X,\mathbb{Z}(j)).$$

\begin{thm}
For any short exact sequence of holomorphic vector bundles on $X$
$$0\rightarrow E' \rightarrow E \rightarrow E''\rightarrow 0$$
one has $\hat{d}(E)=\hat{d}(E')*\hat{d}(E'')$.
\end{thm}

\begin{proof}
We have $E'\oplus E''\cong E$ as smooth bundles, so we consider them as the
same bundle with different holomorphic structures. The purpose is to show these
two holomorphic bundles have the same total Chern class valued in
Deligne cohomology. The idea of the proof is the following.
We fix a hermitian metric on this smooth bundle, choose local holomorphic bases for
those two holomorphic structures respectively, and calculate the hermitian connections
with respect to them. Then we calculate the smooth transgression form which
represents the difference of Cheeger-Simons Chern classes
of these two holomorphic bundles, and show that under the projection
$\Pi_k$, this transgression form represents a zero spark class in
$\hat{\mathbf{H}}^{2k-1}(X,k)$.
Hence $\hat{d}(E)=\hat{d}(E'\oplus E'')=\hat{d}(E')*\hat{d}(E'')$.

Choose a $C^{\infty}$-splitting \xymatrix{ & 0 \ar[r] & E'
\ar[r]^{i} & E \ar@<.7ex>[r]^{\pi}  & E''\ar@<.7ex>[l]^{\sigma} \ar[r] & 0.}

Fix hermitian metrics $h_1$ and $h_2$ for $E'$ and $E''$ respectively, and
define a hermitian metric $h=h_1\oplus h_2$ on $E$ via the smooth isomorphism
$(i,\sigma):E'\oplus E''\rightarrow E.$

Over a small open set $U\subset X$, we choose a local holomorphic basis
$\{e_1,e_2,...,e_m\}$ for $E'$ and a local holomorphic basis
$\{e_{m+1},e_{m+2},...,e_{m+n}\}$ for $E''$. Then choose a local holomorphic basis
$\{\tilde{e}_1,\tilde{e}_2,...,\tilde{e}_m,\tilde{e}_{m+1},\tilde{e}_{m+2},...,\tilde{e}_{m+n}\}$
for $E$ such that $\tilde{e}_i=e_i$ for $1\leq i\leq m$ and $\tilde{e}_{m+j}$
is a holomorphic lift of $e_{m+j}$ for $1\leq j\leq n$. Assume $g=(g_{ij})$ is the transition
matrix for these two bases, i.e. $\tilde{e}_i=\sum_{j=1}^{m+n}g_{ij}e_j$. Then it is easy to know
$g$ has the form
\[ \begin{matrix}
g=
\begin{pmatrix} I_m & 0\\A &I_n \end{pmatrix}
& \text{ and } & g^{-1}=
\begin{pmatrix} I_m & 0\\-A &I_n \end{pmatrix}
\end{matrix} \]
where $I_m$ and $I_n$ is the identity matrices of rank $m$ and $n$, and $A$ is the
nontrivial part of $g$.

Let $H_1$ and $H_2$ be the hermitian matrices representing the metrics $h_1$
and $h_2$ with respect to the bases $\{e_1,e_2,...,e_m\}$ and
$\{e_{m+1},e_{m+2},...,e_{m+n}\}$. Let $H$ and $\tilde{H}$ be the hermitian
matrices representing the metric $h$ with respect to bases $\{e_i\}_{i=1}^{m+n}$ and
$\{\tilde{e}_i\}_{i=1}^{m+n}$. Then we have
\[ \begin{matrix}
H=
\begin{pmatrix} H_1 & 0\\0 &H_2 \end{pmatrix}
& \text{ and } & \tilde{H}=
gHg^*
\end{matrix} \]
where $g^*=\bar{g}^{t}$ is the transpose conjugate of $g$.

Fix the hermitian metric $h$,  we calculate the canonical hermitian connections
with respect to two holomorphic structures. For $E'\oplus E''$, the hermitian
connection $\nabla_0$ can be written locally as the matrix ( w.r.t. the basis $\{e_i\}$ )
$$\theta_0=\partial H\cdot H^{-1}.$$
For $E$, the hermitian connection $\nabla_1$ can be written locally
as the matrix ( w.r.t. the basis $\{\tilde{e}_i\}$ )
$$\tilde{\theta}_1=\partial \tilde{H}\cdot \tilde{H}^{-1}=\partial(gHg^*)(gHg^*)^{-1}
=\partial g\cdot g^{-1}+g\partial H\cdot H^{-1}g^{-1}+gH\partial g^* (g^*)^{-1}H^{-1}g^{-1}.$$
We change the basis and write $\nabla_1$ as the matrix with respect to
the basis $\{e_i\}$
\begin{eqnarray*}
\theta_1&=&d(g^{-1})\cdot g+g^{-1}\tilde{\theta}_1g\\
&=&-g^{-1}dg+g^{-1}(\partial g\cdot g^{-1}+g\partial H\cdot H^{-1}g^{-1}+gH\partial g^* (g^*)^{-1}H^{-1}g^{-1})g\\
&=&-g^{-1}dg+g^{-1}\partial g+\partial H\cdot H^{-1}+H\partial g^*(g^*)^{-1}H^{-1}\\
&=&-g^{-1}\bar{\partial}g+\theta_0+H\partial g^*(g^*)^{-1}H^{-1}\\
&=&\theta_0+H\partial g^*(g^*)^{-1}H^{-1}-g^{-1}\bar{\partial}g
\end{eqnarray*}

Let $\eta\equiv\theta_1-\theta_0=H\partial g^*(g^*)^{-1}H^{-1}-g^{-1}\bar{\partial}g$ and
$\eta^{1,0}=H\partial g^*(g^*)^{-1}H^{-1}$, $\eta^{0,1}=-g^{-1}\bar{\partial}g$ be the
$(1,0)$ and $(0,1)$ components of $\eta$ respectively. Then we have
\[ \begin{matrix}
\eta^{1,0}=
&\begin{pmatrix} H_1 & 0\\0 &H_2 \end{pmatrix}
\begin{pmatrix} 0 & \partial A^*\\0 &0 \end{pmatrix}
\begin{pmatrix} I_m & -A^*\\0 &I_n \end{pmatrix}
\begin{pmatrix} H_1^{-1} & 0\\0 &H_2^{-1} \end{pmatrix}
&=
\begin{pmatrix} 0 & H_1\partial A^*H_2^{-1}\\0 &0 \end{pmatrix}
\end{matrix} \]
and
\[ \begin{matrix}
\eta^{0,1}=
&-\begin{pmatrix} I_m & 0\\-A &I_n \end{pmatrix}
\begin{pmatrix} 0 & 0\\ \bar{\partial}A &0 \end{pmatrix}
&=
-\begin{pmatrix} 0 & 0\\\bar{\partial}A &0 \end{pmatrix}
\end{matrix}. \]
Define a family of connections $\nabla_t$ with
connection matrices $\theta_t=\theta_0+t\eta$ for $0\leq t\leq 1$. Let
$\Omega_t=d\theta_t-\theta_t\wedge\theta_t$ be the curvature of the connection $\theta_t$.
It is easy to see
$$\Omega_t^{0,2}=t\bar{\partial}\eta^{0,1}-t^2(\eta^{0,1}\wedge\eta^{0,1})=0$$
and
$$\Omega_t^{1,1}=\bar{\partial}\theta_0+t(\bar{\partial}\eta^{1,0}+\partial\eta^{0,1}
-\theta_0\wedge\eta^{0,1}-\eta^{0,1}\wedge\theta_0)-t^2(\eta^{1,0}\wedge\eta^{0,1}+
\eta^{0,1}\wedge\eta^{1,0}).$$
Note that
\[ \begin{matrix}
\eta^{0,1}\wedge\eta^{0,1}=
&\begin{pmatrix} 0 & 0\\-\bar{\partial}A &0 \end{pmatrix}\wedge
\begin{pmatrix} 0 & 0\\ -\bar{\partial}A &0 \end{pmatrix}
&=0
\end{matrix}. \]
We will use this trick again in the later calculation.

Suppose that $\Phi$ is an symmetric invariant $k$-multilinear function on the
Lie algebra $\mathfrak{gl}_{m+n}(\mathbb{C})$. Then the two connections
$\nabla_0$ and $\nabla_1$ on $E$ give rise to two Cheeger-Simons differential characters
$\hat{\Phi}_0$ and $\hat{\Phi}_1$, and the difference
$$\hat{\Phi}_0-\hat{\Phi}_1=[T_{\Phi}]$$
where $[T_{\Phi}]$ is the character associated to the smooth form
$$T_{\Phi}=k\int_0^1\Phi(\eta,\Omega_t,\Omega_t,...,\Omega_t)dt.$$

Our goal is to show that $\Pi_k([T_{\Phi}])=[\pi_k(T_{\Phi})]$ represents a
zero spark class. So it suffices to show $\pi_k(T_{\Phi})$ is a
$d_k$-exact form. In fact, we shall show $\pi_k(T_{\Phi})$ is a form of pure type $(k-1,k)$
and equals $\bar{\partial}S=d_kS$ for some $(k-1,k-1)$ form $S$.

\begin{lem}
$T^{i,2k-1-i}_{\Phi}=0$ for $i<k-1$, i.e. $\pi_k(T_{\Phi})=T^{k-1,k}_{\Phi}$, where
$T^{i,2k-1-i}_{\Phi}$ is the $(i,2k-1-i)$ Hodge component of $T_{\Phi}$.
\end{lem}
\begin{proof}
Note that we have $\Omega_t^{0,2}=0$, i.e. $\Omega_t$ is of type $(1,1)$ and $(2,0)$.
Hence it is easy to see $T_{\Phi}^{i,2k-1-i}=0$ for  $i<k-1$ from the
expression $T_{\Phi}=k\int_0^1\Phi(\eta,\Omega_t,\Omega_t,...,\Omega_t)dt$.
\end{proof}

In order to show $T_{\Phi}$ is $\bar{\partial}$-exact for general $\Phi$, we
first show $T_{\Psi_k}$ is $\bar{\partial}$-exact for $\Psi_k(A_1,A_2,...,A_k)=
tr(A_1\cdot A_2\cdot ... \cdot A_k)$.
\begin{lem}
Let $\Psi_k(A_1,A_2,...,A_k)=tr(A_1\cdot A_2\cdot ... \cdot A_k)$ and
$T=T_{\Psi_k}=k\int_0^1tr(\eta\wedge(\Omega_t)^{k-1})dt$. Then
$T^{k-1,k}$ is $\bar{\partial}$-exact. Explicitly, $T^{0,1}=0$ when $k=1$, and for $k\geq2$,
$$T^{k-1,k}=k\int_0^1tr(\eta^{0,1}\wedge(\Omega_t^{1,1})^{k-1})dt
=k\bar{\partial}\int_0^1-tr(\eta^{0,1}\wedge\eta^{1,0}\wedge
(\Omega_t^{1,1})^{k-2})\cdot tdt.$$
\end{lem}
\begin{proof}
When $k=1$, $T^{0,1}=\int_0^1tr(\eta^{0,1})dt=0$ since $\eta^{0,1}$ has the form
\[\begin{pmatrix} 0 & 0\\-\bar{\partial}A &0 \end{pmatrix}.\]

When $k\geq2$, it is easy to see
$$T^{k-1,k}=k\int_0^1tr(\eta^{0,1}\wedge(\Omega_t^{1,1})^{k-1})dt$$ by comparing Hodge components on
both sides. So it suffices to show the identity
$$tr(\eta^{0,1}\wedge(\Omega_t^{1,1})^{k-1})=-\bar{\partial}
tr(\eta^{0,1}\wedge\eta^{1,0}\wedge (\Omega_t^{1,1})^{k-2})\cdot t.$$

First, we introduce some basic identities.
We know that in our theory, Chern classes in Deligne cohomology
are independent of the choice of hermitian metric, and the question above is local.
So we fix local bases and choose hermitian metrics $h_1$ and $h_2$ such that $H_1=I_m$ and $H_2=I_n$ locally.
Now we have
\[ \begin{matrix}
\eta^{1,0}
&=
\begin{pmatrix} 0 & H_1\partial A^*H_2^{-1}\\0 &0 \end{pmatrix}
&=
\begin{pmatrix} 0 & \partial A^*\\0 &0 \end{pmatrix}
&\text{ and }
&\eta^{0,1}
&=\begin{pmatrix} 0 & 0\\-\bar{\partial}A &0 \end{pmatrix}
\end{matrix} \]

$$\Omega_t^{1,1}=t(\bar{\partial}\eta^{1,0}+\partial\eta^{0,1})-
t^2(\eta^{1,0}\wedge\eta^{0,1}+\eta^{0,1}\wedge\eta^{1,0})
=td\eta-t^2\eta\wedge\eta.$$
Note in the equation above, we use the fact $\eta^{1,0}$ is $\partial$-exact,
$\eta^{0,1}$ is $\bar{\partial}$-exact, and $\eta^{1,0}\wedge\eta^{1,0}=0$,
$\eta^{0,1}\wedge\eta^{0,1}=0$ by matrix multiplication.

By calculation, we have
$$[\Omega_t^{1,1},\eta]\equiv\Omega_t^{1,1}\wedge\eta-\eta\wedge\Omega_t^{1,1}=
td(\eta\wedge\eta).$$
Take $(2,1)$ and $(1,2)$ components respectively, we have
$$[\Omega_t^{1,1},\eta^{1,0}]=t\partial(\eta\wedge\eta) \quad\text{and}\quad
[\Omega_t^{1,1},\eta^{0,1}]=t\bar{\partial}(\eta\wedge\eta).$$

The next observation is
$$\bar{\partial}\Omega_t^{1,1}=-t^2\bar{\partial}(\eta\wedge\eta)=
-t[\Omega_t^{1,1},\eta^{0,1}].$$
Using identities above, it is easy to conclude
$$\bar{\partial}((\Omega_t^{1,1})^n)=
-t[(\Omega_t^{1,1})^n,\eta^{0,1}].$$

Now we are ready to calculate.
\begin{eqnarray*}
& &-\bar{\partial}tr(\eta^{0,1}\wedge\eta^{1,0}\wedge (\Omega_t^{1,1})^n)\cdot t\\
&=&-t\cdot tr(\bar{\partial}(\eta^{0,1}\wedge\eta^{1,0}\wedge (\Omega_t^{1,1})^{n}))\\
&=&-t\cdot tr(-\eta^{0,1}\wedge\bar{\partial}\eta^{1,0}\wedge (\Omega_t^{1,1})^{n}+
\eta^{0,1}\wedge\eta^{1,0}\wedge \bar{\partial}(\Omega_t^{1,1})^{n})\\
&=&-t\cdot tr(-\eta^{0,1}\wedge\bar{\partial}\eta^{1,0}\wedge (\Omega_t^{1,1})^{n}+
\eta^{0,1}\wedge\eta^{1,0}\wedge (-t)[(\Omega_t^{1,1})^n,\eta^{0,1}])\\
&=&t\cdot tr(\eta^{0,1}\wedge\bar{\partial}\eta^{1,0}\wedge (\Omega_t^{1,1})^{n}+t
\eta^{0,1}\wedge\eta^{1,0}\wedge (\Omega_t^{1,1})^n\wedge \eta^{0,1}-
t\eta^{0,1}\wedge\eta^{1,0}\wedge \eta^{0,1}\wedge (\Omega_t^{1,1})^n)\\
&=&t\cdot tr(\eta^{0,1}\wedge\bar{\partial}\eta^{1,0}\wedge (\Omega_t^{1,1})^{n}-
t\eta^{0,1}\wedge\eta^{1,0}\wedge \eta^{0,1}\wedge (\Omega_t^{1,1})^n)+t^2\cdot
tr(\eta^{0,1}\wedge\eta^{1,0}\wedge (\Omega_t^{1,1})^n\wedge \eta^{0,1})\\
&=&tr(\eta^{0,1}\wedge(t\bar{\partial}\eta^{1,0}-t^2\eta^{1,0}\wedge \eta^{0,1})
\wedge (\Omega_t^{1,1})^{n})+t^2\cdot
tr(\eta^{0,1}\wedge \eta^{0,1}\wedge\eta^{1,0}\wedge (\Omega_t^{1,1})^n)\\
&\stackrel{*}{=}&tr(\eta^{0,1}\wedge\Omega_t^{1,1}\wedge (\Omega_t^{1,1})^{n})+0\\
&=&tr(\eta^{0,1}\wedge  (\Omega_t^{1,1})^{n+1})
\end{eqnarray*}
Put $n=k-2$, we are done.

Note that in the second to last equality, we use the trick $\eta^{0,1}\wedge \eta^{0,1}=0$
several times.

\end{proof}

Recall in the Chern-Weil theory, the $k$th Chern character of a vector bundle is
represented by the form $\frac{1}{k!}\Psi_k(\Omega,\Omega,...,\Omega)=\frac{1}{k!}tr(\Omega^k)$
where $\Omega$ is the curvature of any connection.
Any symmetric invariant $k$-multilinear function $\Phi$ on the
Lie algebra $\mathfrak{gl}_{m+n}(\mathbb{C})$ is generated by $\Psi_1$, $\Psi_2$, ... $\Psi_k$,
i.e. we have $$\Phi=\sum_{n=1}^k\sum_{i_1+\cdots+i_n=k \atop i_1>0,\cdots,i_n>0}
a_{i_1i_2\cdots i_n}\Psi_{i_1}\otimes\Psi_{i_2}\otimes...\otimes\Psi_{i_n}.$$

Hence, $T_{\Phi}=k\int_0^1\Phi(\eta,\Omega_t,\Omega_t,...,\Omega_t)dt$ where
$\Phi(\eta,\Omega_t,\Omega_t,...,\Omega_t)$ is a sum with summands like
$\Psi_{i_1}(\eta,\Omega_t,\Omega_t,...,\Omega_t)\Psi_{i_2}(\Omega_t,\Omega_t,...,\Omega_t)
...\Psi_{i_n}(\Omega_t,\Omega_t,...,\Omega_t)$. For $j>1$, $\Psi_{i_j}(\Omega_t,\Omega_t,...,\Omega_t)$
is a closed $(i_j,i_j)$ form representing $i_j!$ times the $i_j$th Chern character.
And from last lemma, we know $\Psi_{i_1}(\eta,\Omega_t,\Omega_t,...,\Omega_t)$ has types
$(i_1-1,i_1)$ and higher, and its $(i_1-1,i_1)$ component is $\bar{\partial}$-exact.
Therefore, $T_{\Phi}$ is of types $(k-1,k)$ and higher, and
$\pi_k(T)=T^{k-1,k}=\bar{\partial}S^{k-1,k-1}$ for some $(k-1,k-1)$
form $S^{k-1,k-1}$.
\end{proof}

\begin{rmk}
By the theorems on uniqueness of Chern classes in Deligne cohomology in \cite{B}, \cite{EV} and
\cite{Br2}, our theory on Chern classes is equivalent to all others. In particular, when the setting
is algbraic, Chern classes defined above are compatible with Grothendieck-Chern classes under the
cycle map $\psi: CH^*(X)\rightarrow H^{2*}_{\mathcal{D}}(X,\mathbb{Z}(*))$.
\end{rmk}

Cheeger and Simons also defined Chern characters for vector bundles with connections,
which are located in rational differential characters $\hat{H}^*(X,\mathbb{R}/\mathbb{Q})$.
For holomorphic vector bundles, we can project Chern characters in differential characters to get
Chern characters in rational Deligne cohomology $H^{2*}_{\mathcal{D}}(X,\mathbb{Q}(*))$. Define
$$\widehat{dch}_k(E)\equiv\Pi_k(\hat{ch}_k(E,\nabla))\in H^{2k}_{\mathcal{D}}(X,\mathbb{Q}(k)),$$
where $\nabla$ is the hermitian connection associated to a hermitian metric.

Since $\hat{ch}(E\oplus E',\nabla\oplus\nabla')=\hat{ch}(E,\nabla)+\hat{ch}(E',\nabla')$,
we have
\begin{thm}
If $E$ and $F$ be two holomorphic vector bundles on complex manifold $X$,
then $$\widehat{dch}(E\oplus F)=\widehat{dch}(E)+\widehat{dch}(F).$$
\end{thm}
Moreover,
\begin{thm}
For any short exact sequence of holomorphic vector bundles on $X$
$$0\rightarrow E' \rightarrow E \rightarrow E''\rightarrow 0$$
one has $\widehat{dch}(E)=\widehat{dch}(E')+\widehat{dch}(E'')$.
\end{thm}

\section{Bott Vanishing for Holomorphic Foliations}
In \cite{Bo}, Bott constructed a family of connections on the normal bundle of any smooth
foliation of a manifold and established the Bott vanishing theorem which says
the characteristic classes of the normal bundle are trivial in all sufficiently high degrees.
We shall show an analogue of the Bott vanishing theorem for Chern classes of the normal
bundle of a holomorphic foliation.

Suppose that $N$ is the normal bundle to a holomorphic foliation of codimension $p$ on
a complex manifold $X$. Then there are two natural families of connections to consider
on $N$, the family of Bott connections and the family of canonical hermitian connections.

\begin{prop}
$N$ is a holomorphic vector bundle on $X$ as above.
Let $P(c_1,...,c_q)$ be a polynomial in Chern classes which is of pure degree $2k$
with $k>2q$. Then the projection image of Cheeger-Simons Chern class $\Pi_k(P(\hat{c_1},...,\hat{c_q})) \in
\hat{H}^{2k-1}(X,k)$ for Bott connections agrees with the Chern class in Deligne cohomology
$P(\hat{d}_1,...,\hat{d}_q)$ for the canonical hermitian connections.

\end{prop}

\begin{proof}
In fact, this is a direct corollary of Remark 8.2 since Bott connections are compatible with
the holomorphic structure of $N$.

Let $\nabla$ be a Bott connection, $\tilde{\nabla}$ be the canonical hermitian
connection for some hermitian metric and $\theta$, $\tilde{\theta}$ be their
connection forms. Notice that both $\theta$, $\tilde{\theta}$ are of type
$(1,0)$.

Let $\Phi(X_1,...,X_k)$ be the symmetric invariant $k$-multilinear function
on the Lie algebra $\mathfrak{gl}_q(\mathbb{C})$ such that
$P(\sigma_1(X),...,\sigma_q(X))=\Phi(X,...,X)$ where $\sigma_j$ is the $j^{th}$
elementary symmetric function of the eigenvalues of $X$. Then the difference
between the Cheeger-Simons Chern class associated to $P$ for the two
connections $\nabla$ and $\tilde{\nabla}$ is the character associated to the
smooth form
$$T=k\int_0^1 \Phi(\theta-\tilde{\theta},\Omega_t,...,\Omega_t)dt$$ where
$\Omega_t=d\theta_t-\theta_t\wedge\theta_t$. Since $\theta_t$ is of type
$(1,0)$ and $\Omega_t$ is of type $(1,1)$ and $(2,0)$, we have
$T^{p,q}=0$ for all $p<q$.

Therefore,
\begin{eqnarray*}
& &P(\hat{d}_1,...,\hat{d}_q)-\Pi_k(P(\hat{c}_1(N,\nabla),...,\hat{c}_q(N,\nabla)))\\
&=&\Pi_k(P(\hat{c}_1(N,\tilde{\nabla}),...,\hat{c}_q(N,\tilde{\nabla})))-
\Pi_k(P(\hat{c}_1(N,\nabla),...,\hat{c}_q(N,\nabla)))\\
&=&\Pi_k(T)=0
\end{eqnarray*}
\end{proof}

\begin{thm}
Let $N$ be a holomorphic bundle of rank $q$ on a complex manifold $X$. If $N$
is (isomorphic to) the normal bundle of a holomorphic foliation of $X$, then
for every polynomial $P$ of pure degree $k>2q$, the associated refined
Chern class satisfies
$$P(\hat{d}_1(N),...,\hat{d}_q(N)) \in
Im [H^{2k-1}(X, \mathbb{C}^{\times}) \rightarrow
H_{\mathcal{D}}^{2k}(X,\mathbb{Z}(k))].$$
\end{thm}

\begin{proof}
We have the following commutative diagram:

 \xymatrix{& 0 \ar[r] & H^{2k-1}(X,\mathbb{C}^{\times}) \ar[r] \ar[d]
 & \hat{H}^{2k-1}(X) \ar[r] \ar[d] & \mathcal{Z}_{\mathbb{Z}}^{2k}(X) \ar[r] \ar[d] & 0 \\
& 0 \ar[r] & H_{\mathcal{D}}^{2k}(X,\mathbb{Z}(k)) \ar[r] & \hat{H}^{2k-1}(X,k) \ar[r] &
\mathcal{Z}_{\mathbb{Z}}^{2k}(X,k) \ar[r] & 0}

We know $P(\hat{d}_1(N),...,\hat{d}_q(N)) \in H_{\mathcal{D}}^{2k}(X,\mathbb{Z}(k))$ and
$P(\hat{c_1},...,\hat{c_q}) \in \hat{H}^{2k-1}(X)$. By last proposition, we know they
have the same images in $\hat{H}^{2k-1}(X,k)$. And by Bott vanishing theorem, we have
$P(\hat{c_1},...,\hat{c_q}) \in H^{2k-1}(X,\mathbb{C}^{\times})$. Then we get the
conclusion.
\end{proof}

\section{Nadel Invariants for Holomorphic Vector Bundles}
In his beautiful paper \cite{N}, Nadel introduced interesting relative invariants for holomorphic
vector bundles. Explicitly, for two holomorphic vector bundles $E$ and $F$ over a
complex manifold $X$ and a $C^{\infty}$ isomorphism $f:E\rightarrow F$, Nadel
defined invariants $\mathscr{E}^k(E,F,f)\in H^{2k-1}(X,\mathcal{O})$ and
$\mathscr{E}^k(E,F)\in H^{2k-1}(X,\mathcal{O})/H^{2k-1}(X,\mathbb{Z})$.
He also conjectured that these invariants should coincide with a component
of the Abel-Jacobi image of $k!(ch_k(E)-ch_k(F))\in CH^k(X)$ when the setting is algebraic.
In \cite{Be}, Berthomieu developed a relative K-theory and gave an affirmative answer to
Nadel's conjecture.

In this section, we shall generalize Nadel theory and give a short proof of his conjecture.
From our point of view, if $E$ and $F$ are two holomorphic
vector bundles whose underlying $C^{\infty}$ vector bundles are isomorphic, then
their usual Chern classes coincide, and the difference
of their $k$th Chern classes in Deligne cohomology $\hat{d}_k(E)-\hat{d}_k(F)$
is located in the intermediate Jacobians $\mathcal{J}^k(X)$. Hence we can define relative
invariants for a pair $(E,F)$ directly in $\mathcal{J}^k(X)$.
In particular, we shall express the difference of $k$th Chern character
$\widehat{dch}_k(E)-\widehat{dch}_k(F)$ by transgression forms whose
$(0,2k-1)$ components are exactly the Nadel invariants.
This will prove Nadel's conjecture in more general context ( not necessarily algebraic ).

Let $E$ and $F$ be two holomorphic vector bundles over complex manifold $X$, and
$g:E \rightarrow F$ be a $C^{\infty}$ bundle isomorphism. Fix a hermitian metric $h$
for $E$ and $F$. Over a small open set $U\subset X$, choose local holomorphic bases $\{e_i\}_{i=1}^{r}$
and $\{f_i\}_{i=1}^{r}$ for $E$ and $F$ respectively, and denote also by $g$ the transition
matrix of the $C^{\infty}$ bundle isomorphism with respect to bases $\{e_i\}_{i=1}^{r}$
and $\{f_i\}_{i=1}^{r}$. Let $H$ and $\tilde{H}$ be the hermitian
matrices representing the metric $h$ with respect to bases $\{e_i\}_{i=1}^{r}$
and $\{f_i\}_{i=1}^{r}$. Then we have $\tilde{H}=gHg^*$.

Now we calculate the canonical hermitian connections
with respect to two holomorphic structures. For $E$, the hermitian
connection $\nabla_0$ can be written locally as the matrix ( w.r.t. the basis $\{e_i\}$ )
$$\theta_0=\partial H\cdot H^{-1}.$$
For $F$, the hermitian connection $\nabla_1$ can be written locally
as the matrix ( w.r.t. the basis $\{f_i\}$ )
$$\tilde{\theta}_1=\partial \tilde{H}\cdot \tilde{H}^{-1}=\partial(gHg^*)(gHg^*)^{-1}
=\partial g\cdot g^{-1}+g\partial H\cdot H^{-1}g^{-1}+gH\partial g^* (g^*)^{-1}H^{-1}g^{-1}.$$
We change the basis and write $\nabla_1$ as the matrix with respect to
the basis $\{e_i\}$
\begin{eqnarray*}
\theta_1&=&d(g^{-1})\cdot g+g^{-1}\tilde{\theta}_1g\\
&=&-g^{-1}dg+g^{-1}(\partial g\cdot g^{-1}+g\partial H\cdot H^{-1}g^{-1}+gH\partial g^* (g^*)^{-1}H^{-1}g^{-1})g\\
&=&-g^{-1}dg+g^{-1}\partial g+\partial H\cdot H^{-1}+H\partial g^*(g^*)^{-1}H^{-1}\\
&=&\theta_0+H\partial g^*(g^*)^{-1}H^{-1}-g^{-1}\bar{\partial}g
\end{eqnarray*}

Let $\eta=\theta_1-\theta_0=H\partial g^*(g^*)^{-1}H^{-1}-g^{-1}\bar{\partial}g$ and
$\eta^{1,0}=H\partial g^*(g^*)^{-1}H^{-1}$, $\eta^{0,1}=-g^{-1}\bar{\partial}g$ be the
$(1,0)$ and $(0,1)$ components of $\eta$ respectively. Define a family of connections
$\nabla_t$ with connection matrices $\theta_t=\theta_0+t\eta$ for $0\leq t\leq 1$. Let
$\Omega_t=d\theta_t-\theta_t\wedge\theta_t$ be the curvature of the connection $\theta_t$.
$$\Omega_t^{0,2}=t\bar{\partial}\eta^{0,1}-t^2(\eta^{0,1}\wedge\eta^{0,1})=
t\bar{\partial}(-g^{-1}\bar{\partial}g)-t^2(-g^{-1}\bar{\partial}g)\wedge(-g^{-1}\bar{\partial}g)
=(t-t^2)g^{-1}\bar{\partial}g\wedge g^{-1}\bar{\partial}g.$$

Suppose that $\Phi$ is an symmetric invariant $k$-multilinear function on the
Lie algebra $\mathfrak{gl}_{m+n}(\mathbb{C})$. Then the two connections
$\nabla_0$ and $\nabla_1$ on $E$ give rise to two Cheeger-Simons differential characters
$\hat{\Phi}_0$ and $\hat{\Phi}_1$, and the difference
$$\hat{\Phi}_0-\hat{\Phi}_1=[T_{\Phi}]$$
where $[T_{\Phi}]$ is the character associated to the smooth form
$$T_{\Phi}=k\int_0^1\Phi(\eta,\Omega_t,\Omega_t,...,\Omega_t)dt.$$

The difference of Chern classes $\Pi_k([T_{\Phi}])=\Pi_k(\hat{\Phi}_0-\hat{\Phi}_1)\in \mathcal{J}^k$
is a spark class which is represented by the smooth form $\pi_k(T_{\Phi})$.
In particular, when $\Phi$ is the form $\Phi(A_1,A_2,...,A_k)=tr(A_1\cdot A_2\cdot...\cdot A_k)$
( representing $k!$ times the $k$th Chern character ),
we have
$$T_{\Phi}=k\int_0^1tr(\eta\wedge(\Omega_t)^{k-1})dt \qquad \text{and} \qquad
\pi_k(T_{\Phi})=\pi_k(k\int_0^1tr(\eta\wedge(\Omega_t)^{k-1})dt).$$
The $(0,2k-1)$ component
\begin{eqnarray*}
\pi_k(T_{\Phi})^{0,2k-1}&=&T_{\Phi}^{0,2k-1}\\
&=&k\int_0^1tr(\eta^{0,1}\wedge(\Omega_t^{0,2})^{k-1})dt\\
&=&k\int_0^1tr(-g^{-1}\bar{\partial}g\wedge((t-t^2)g^{-1}\bar{\partial}g\wedge g^{-1}\bar{\partial}g)^{k-1})dt\\
&=&k\int_0^1-t^{k-1}(1-t)^{k-1}tr((g^{-1}\bar{\partial}g)^{2k-1})dt\\
&=&-k\int_0^1t^{k-1}(1-t)^{k-1}dt\cdot tr((g^{-1}\bar{\partial}g)^{2k-1})\\
&=&\pm\mathscr{E}^k(E,F,g)
\end{eqnarray*}

\begin{defn}
$X$ is a complex manifold, $E$ and $F$ are two holomorphic vector bundles over $X$ whose underlying
$C^{\infty}$ vector bundle are isomorphic. Define Nadel-type invariants
$$\hat{\mathscr{E}}^k(E,F)\equiv k!(\widehat{dch}_k(E)-\widehat{dch}_k(F))=
[\pi_k(k\int_0^1tr(\eta\wedge(\Omega_t)^{k-1})dt)]\in \mathcal{J}^k.$$
\end{defn}

Note that we have a natural projection
$$\pi:\mathcal{J}^{k}\equiv H^{2k-1}(X,\mathbb{C})/F^{k}H^{2k-1}(X,\mathbb{C})\oplus H^{2k-1}(X,\mathbb{Z})
\rightarrow H^{2k-1}(X,\mathcal{O})/H^{2k-1}(X,\mathbb{Z}).$$
By the calculations above, we have
\begin{thm}
The $k$th Nadel invariant $\mathscr{E}^k(E,F)$ is the image of $\hat{\mathscr{E}}^k(E,F)$
under the map $\pi:\mathcal{J}^k\rightarrow H^{2k-1}(X,\mathcal{O})/H^{2k-1}(X,\mathbb{Z})$.
That is, Nadel's conjecture is true.
\end{thm}

\bigskip

\noindent {\sc Ning Hao }\\
Mathematics Department, SUNY at Stony Brook\\
Stony Brook, NY 11794\\
Email: {\tt nhao@math.sunysb.edu}
\end{document}